\documentclass[11pt]{amsart}
\usepackage[margin=1.6in]{geometry}
\usepackage[dvips]{graphicx}
\usepackage{abstract}
\usepackage[labelfont=bf]{caption}
\usepackage{amscd}
\usepackage{comment}
\usepackage{amsfonts}
\usepackage{color}
\usepackage{multicol}
\usepackage[all]{xy}
\usepackage{amsmath,amssymb,amsthm,latexsym}
\usepackage{amscd}
\usepackage{titlesec}
\titleformat {\title}
  {\normalfont \HUGE \bfseries \raggedright}{\thetitle}{1em}{}
\titleformat {\section}
  {\normalfont \bfseries \raggedright \center}{\thesection}{1em}{}
\titleformat {\subsection}[runin]
  {\normalfont  \bfseries \raggedright}{\thesubsection}{1em}{}
\titleformat {\subsubsection}[runin]
  {\normalfont  \bfseries \raggedright}{\thesubsubsection}{1em}{}
\titleformat {\abstract}
  {\normalfont  \bfseries \raggedright}{\theabstract}{1em}{}
		
\makeatletter
\renewcommand*{\@cite@ofmt}{\bfseries\hbox}
\g@addto@macro\bfseries{\boldmath}
\makeatother

\numberwithin{equation}{section}
\newtheorem{theorem}{Theorem}[section]
\newtheorem{proposition}[theorem]{Proposition}
\newtheorem{definition}[theorem]{Definition}
\newtheorem{corollary}[theorem]{Corollary}
\newtheorem{lemma}[theorem]{Lemma}

\theoremstyle{remark}
\newtheorem{remark}[theorem]{\bf Remark}
\newtheorem{example}[theorem]{\bf Example}

\newcommand{\SSS}{\mathbb S}
\newcommand{\HHH}{\mathbb H}
\newcommand{\real}{\mathbb R}
\newcommand{\C}{\mathbb C}
\newcommand{\bdm}{\begin{displaymath}}
\newcommand{\edm}{\end{displaymath}}
\newcommand{\beq}{\begin{equation}}
\newcommand{\beqa}{\begin{eqnarray}}
\newcommand{\beqas}{\begin{eqnarray*}}
\newcommand{\eeq}{\end{equation}}
\newcommand{\eeqa}{\end{eqnarray}}
\newcommand{\eeqas}{\end{eqnarray*}}
\newcommand{\dd}{\textup{d}}
\newcommand{\bbar}{\left( \begin{array}}
\newcommand{\ebar}{\end{array} \right)}

\begin{document}
\title[The Bj\"orling problem for Willmore surfaces]{\bf{ On the Bj\"{o}rling problem for Willmore surfaces}}
\author{David Brander, Peng Wang }

\maketitle

\begin{abstract}

 We solve the analogue of Bj\"orling's problem for Willmore surfaces via a harmonic map representation.
For the umbilic-free case the problem and solution are as follows:
given a real analytic curve $y_0$ in $\SSS^3$, together with the prescription
of the values of the surface normal and the dual Willmore surface along the curve,
lifted to the light cone in Minkowski $5$-space $\real^5_1$, we prove,
using isotropic harmonic maps, that there
exists a unique pair of dual Willmore surfaces $y$ and $\hat y$ satisfying the
given values along the curve.  We give explicit formulae for the generalized
Weierstrass data for the surface pair. For the three dimensional
target, we use the solution to explicitly
describe the Weierstrass data, in terms of geometric quantities, for all equivariant Willmore surfaces.
For the case that the surface has umbilic points,
we apply the more general half-isotropic harmonic maps introduced by H\'{e}lein
 to derive a solution: in this case the map $\hat y$ is not necessarily the dual surface, and the
additional data of a derivative of $\hat y$ must be prescribed. This solution is  generalized to higher
codimensions. 
\end{abstract}


\renewcommand{\theenumi}{(\roman{enumi})}
\renewcommand{\labelenumi}{\theenumi}

\section{Introduction}
A \emph{Willmore surface} in Euclidean 3-space $\real^3$ is an immersion $S$ that is
locally critical for the \emph{Willmore functional}
\[
\mathcal{W}(S) = \int_S H^2 \dd A,
\]
where $H$ is the mean curvature of the surface. As such, these surfaces
 are generalizations of
minimal surfaces, and also, from another point of view, of elastic curves.
Hence the interest in Willmore surfaces, which have attracted
a lot of attention in recent decades. The governing equations are a fourth order nonlinear PDE, and they are therefore a challenging class
of surfaces to get information about: for example, the Willmore conjecture,
 that the Clifford torus is the
global minimizer of the Willmore energy among  tori,
 proposed in the 1960's, took more than half a century to resolve \cite{MN}.

 The property of being a Willmore surface
is invariant under conformal transformations of the ambient space.  Hence, from a theoretical point of view, the choice of conformally congruent target space is
unimportant.  In fact the natural choice is the $3$-sphere $\SSS^3$, because
this case includes, up to M\"obius equivalence,
 both $\real^3$ and the hyperbolic space ${\mathbb H}^3$
as proper subspaces. In this article, we  generally regard the surfaces as
living in $\SSS^3$, and more generally $\SSS^n$, $n\geq 3$.
For further introduction and background
on Willmore surfaces, especially relevant to this article, see H\'{e}lein \cite{Helein}.

Being one kind of generalization of minimal surfaces, it is natural to consider
the extension of Bj\"orling's classical problem to Willmore surfaces.
Bj\"orling's problem is to find the unique minimal surface that contains a
given curve with surface normal prescribed along the curve.  The solution
can be found, in terms of the Weierstrass-Enneper representation, via analytic extension of the  prescribed data. It is a useful
tool in the study of minimal surfaces and has been generalized recently,
through various means, to several other surface classes.  An approach that can be expected to be fruitful among surfaces associated
to harmonic maps can be found in the solution for non-minimal constant
mean curvature surfaces given in \cite{Br-Do}. Here one uses an infinite
dimensional version of the Weierstrass-Enneper formula,
the DPW method of Dorfmeister/Pedit/Wu \cite{DPW}, to again obtain the
solution by holomorphic extension.

For Willmore surfaces, there are more than one type of harmonic map one might
consider employing.
For example, it has long been known that the conformal Gauss map into the Grassmannian
$Gr_{3,1}(\mathbb{R}^{5}_{1})$ of Lorentzian $4$-planes in $\real^5_1$ is harmonic.
This is
a certain lift of the surface normal into  $\real^5_1$,
and the harmonicity of this map  has been used in \cite{DoWa1} to study Willmore surfaces
via the DPW method. The related flat connections also form  the basis for
 some of the
recent works on \emph{constrained} Willmore surfaces: see, e.g.
\cite{bq, bohle, flpp, heller2014}.

On the other hand, a different (``roughly'')
 harmonic map, this time into $SO(1,4)/(SO(1,1) \times SO(3))$ was
found by H\'elein in \cite{Helein} (See also \cite{Helein2}).  In our distillation of H\'elein's
work, the basic object is the
map $Y\wedge \hat Y$, where $Y$ and $\hat Y$ are the surface and its
dual, lifted to the light cone.  Essentially, the projections of
$Y$ and $\hat Y$ are Willmore if and only if $Y\wedge \hat Y$
is what we call an \emph{isotropic} harmonic map. The DPW method also works for isotropic harmonic maps, and this is the approach we will use.

\begin{figure}[ht]
\centering
$
\begin{array}{ccc}
\includegraphics[height=28mm]{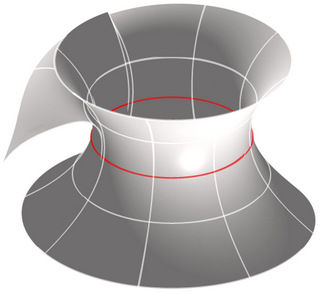} \quad & \quad
\includegraphics[height=28mm]{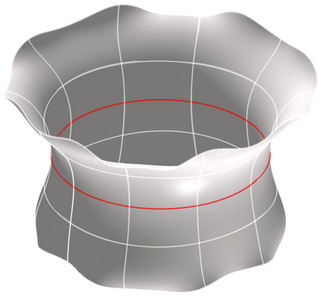} \quad & \quad
\includegraphics[height=28mm]{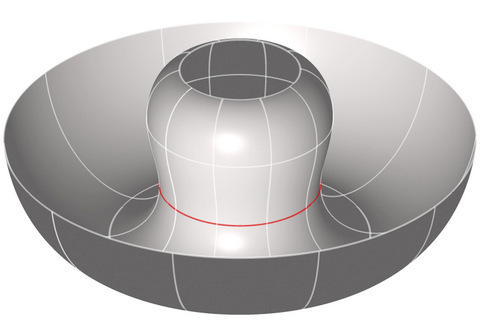}
\end{array}
$
\caption{Three solutions to the Bj\"orling problem for Willmore surfaces in $\SSS^3$,
all with the same initial curve (a circle) and the same  normal along the curve. The
prescribed dual surface data $\hat Y_0$ is different in each case. The surfaces are all given
the same stereographic projection to $\real^3$. }
\label{figure1}
\end{figure}

\subsection{Results of this article}
If only the surface and surface normal are prescribed along a curve, then there
is no hope of obtaining a unique solution for the Bj\"orling problem for Willmore
surfaces (see Figures \ref{figure1} and \ref{figureSOR1}).
One needs to prescribe something more, and it turns out that
the value along the curve of the dual surface $\hat Y$ is enough. Hence, the representation
in terms of $Y \wedge \hat Y$ seems canonical for this problem, rather than the conformal Gauss map representation.

In Section \ref{section2}, we outline the projective light cone model for conformal surface
theory, the basic theory of Willmore surfaces in this setting, and the relation with
isotropic harmonic maps into $SO(1,4)/(SO(1,1) \times SO(3))$.
In Section \ref{section3} we derive the DPW construction for isotropic harmonic
maps.  The DPW construction for harmonic maps $f: \Sigma \to G/K$ makes use of
a \emph{holomorphic frame} $F_-^\lambda$ for the \emph{extended frame}  $F^\lambda: \Sigma \to \Omega G \cong \Lambda G^\C/\Lambda^+G^\C$, a lift of $f$ into the group of based loops
in $G$.  The Maurer-Cartan form $\eta$ of $F_-^\lambda$ is known as a \emph{potential},
and this is the Weierstrass data for the problem. Given a potential
$\eta$, which essentially consists  of a series of arbitrary holomorphic functions,
the equation $\dd F_-^\lambda = F_-^\lambda \eta$ can be solved, and a frame
$F^\lambda: \Sigma \to \Lambda G$ is obtained via the Iwasawa decomposition.
If $G$ is non-compact, all of this happens only on a large open set (the big cell) of
the loop group, but otherwise the theory is the same.  We need to verify that
the theory restricts to \emph{isotropic} harmonic maps (see Definition \ref{isotropicdef}), and this is indeed
the case because the isotropic condition is preserved by the loop group decompositions.

\begin{figure}[ht]
\centering
$
\begin{array}{ccc}
\includegraphics[height=50mm]{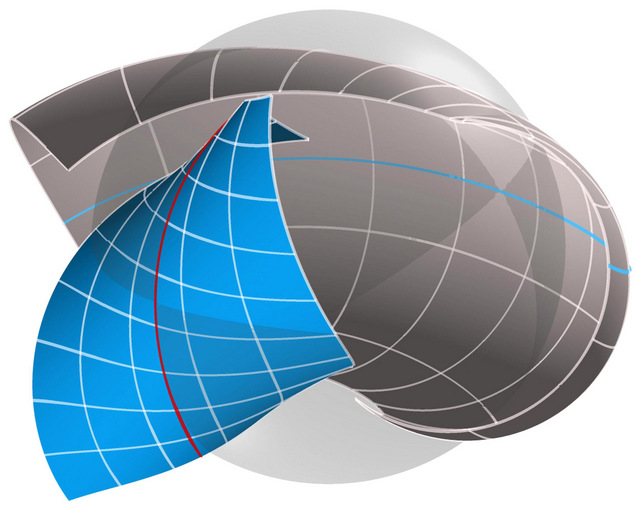} \,\, & \,\,
\end{array}
$
\caption{Dual solutions of Bj\"orling's problem for Willmore surfaces. The prescribed data is the pair of curves (one red, one blue) together with a
family of $2$-spheres tangent to both curves at the touching points. One
sphere is shown.}
\label{figuresc}
\end{figure}

In Section \ref{section4} we present, in Theorems \ref{thm-Bjorling} and
\ref{thm-Bjorling-BP}, a solution to the Bj\"orling problem
for Willmore surfaces: given a real analytic sphere congruence $\psi_0$ (a lift of the surface normal) along a curve $\mathbb I$, with two enveloping curves $Y_0$ and and $\hat Y_0$, there exists a
unique dual pair of Willmore surfaces $Y$ and $\hat Y$ that restrict, along $\mathbb I$,
to $Y_0$ and $\hat Y_0$ (Figure \ref{figuresc}).  We also give an explicit formula for a holomorphic
potential for the surface, in terms of the prescribed geometric data.

In Section \ref{equisection}, we apply this result to describe all $SO(4)-$\emph{equivariant} Willmore surfaces in $\SSS^3$, that is surfaces invariant under the action of a
$1$-parameter subgroup of the isometry group.
 Our approach is
to solve the Bj\"orling problem along a parallel.
One can describe all
 $SO(1,3)-$\emph{equivariant} Willmore surfaces in $\mathbb{H}^3$ in an
analogous way, and we  give the details for some of these, including
hyperbolic rotational surfaces and the hyperbolic analogue of Hopf surfaces in
Section \ref{hypequisection}.
We remark that it is known \cite{bg1986,ls1984} that Willmore surfaces of revolution in $\real^3$ can be obtained by revolving about the $x$-axis an elastic curve in $\HHH^2$, represented by the upper half plane model above the $x$-axis. General equivariant surfaces have not been described so explicitly,
however Ferus and Pedit
\cite{fp1990} gave a description of all non-rotational $SO(4)$-equivariant Willmore tori.

\begin{figure}[ht]
\centering
$
\begin{array}{ccc}
\includegraphics[height=40mm]{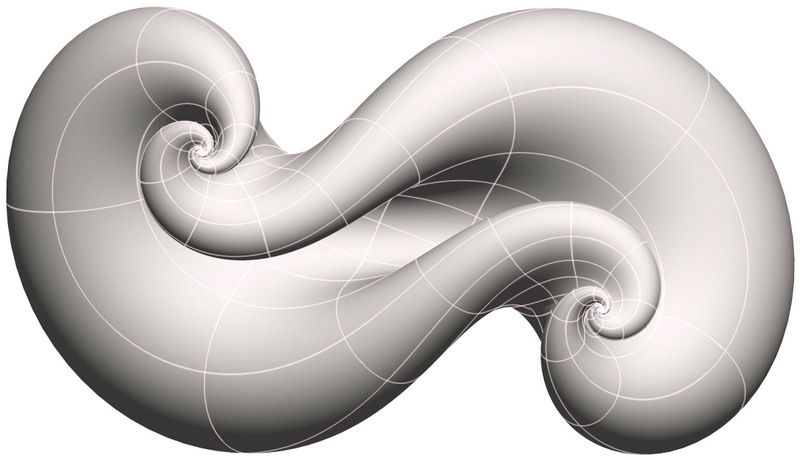} \,\, & \,\,
\end{array}
$
\caption{An $SO(1,3)$-equivariant Willmore surface not congruent to a minimal surface in any space form (Section \ref{sectionhrn0}).}
\label{figure2}
\end{figure}

 In Section \ref{section7} we extend the loop group representation to the case of isotropic
 and  \emph{half}-isotropic harmonic maps for general $n$. The half-isotropic case is a generalization
of the isotropic case where $\hat Y$ is no longer required to be the dual (or geometric adjoint transform)
of $Y$. This section is partly motivated
by the desire to give a uniform treatment of results of H\'elein \cite{Helein},
Xia/Shen \cite{Xia-Shen} and Ma \cite{Ma2006}, but it also allows us to deal with umbilics, which
are ruled out in the isotropic case.

We end this paper in  Section \ref{section8} with an application of the harmonic maps in Section
 \ref{section7}  to the solution of the Bj\"orling problem for Willmore surfaces in $\SSS^{n+2}$.
The half-isotropic setting is needed both for Willmore surfaces in $\SSS^3$ with umbilics and
for general Willmore surfaces in $\SSS^{n+2}$.  Since $\hat Y$ is no longer required to be the dual
of $Y$, there is now more freedom, and so an additional condition is needed to define a unique
solution.  The $v$ derivative, $\hat Y_v$ turns out to be sufficient. 

\subsection{Concluding remarks}
All the images in this article were produced by numerically implementing the DPW method for the problem at hand.  At the time of writing, some code is available at: http://davidbrander.org/software.html. In our examples, mainly working in the isotropic setting, the surfaces appear smooth when the boundary of the Iwasawa
big cell is approached.  One expects that these are
 points where the surface and its dual coincide,
such as can happen at umbilics (see Lemma \ref{umbiliclemma} below). Babich and
Bobenko \cite{bab-bob}, constructed Willmore surfaces which contain lines of
umbilics.  For such solutions, one needs to use the general construction of Section \ref{section8}.

Recently, Jensen, Musso and Nicolodi have provided a solution of the geometric Cauchy
problem for the more general membrane shape equation \cite{jmn2014}.  This equation includes Willmore surfaces as a special case.
Their solution, which needs an umbilic-free assumption, is quite different: the framework is differential systems, the problem is
posed in principal coordinates, the Cauchy data are the curve $y$, the mean curvature $h$ and the transverse derivative $h_v$
along the curve $y(u)$, plus the value of the normal at a single point. Because of these major differences, the range of applications
of their solution is fundamentally different - for example the description of all equivariant surfaces we provide here
does not seem feasible
with their formulation.

\section{Willmore surfaces in $\SSS^{n+2}$}  \label{section2}

\subsection{Conformal surface theory in the projective light cone model}

We will review first the projective light cone model of the conformal geometry of
$\SSS^{n+2}$ and derive the surface theory in this model. Then we formulate it at the Lie algebra level. Our treatment here follows the surface theory in \cite{BPP,Ma}.

We denote the Minkowski space $\mathbb{R}^{n+4}_1$ as $\mathbb{R}^{n+4}$ equipped with a Lorentzian metric
\[
\langle x,y\rangle =-x_{0}y_0+\sum_{j=1}^{n+3}x_jy_j=x^t I_{1,n+3} y,\ \ I_{1,n+3}=\textup{diag}(-1,1,\cdots,1).
\]
Let $\mathcal{C}_+^{n+3}$ be the forward light cone of $\mathbb{R}^{n+4}_{1}$, i.e. for any $x\in\mathcal{C}_+^{n+3}$, $x_0>0$. One can see that the projective light cone
\[
Q^{n+2}=\{\ [x]\in\mathbb{R}P^{n+3}\ |\ x\in \mathcal{C}_+^{n+3}
\},
\]
with the induced conformal metric, is conformally equivalent to $\SSS^{n+2}$, and the conformal group of
$Q^{n+2}$ is exactly the orthogonal group $O(1,n+3)/\{\pm1\}$ of
$\mathbb{R}^{n+4}_1$, acting on $Q^{n+2}$ by
$
T([x])=[Tx],\ T\in O(1,n+3).
$ We denote by $SO^+(1,n+3)$ the connected component of $O(1,n+3)$ containing $I$, that is for any $T\in SO^+(1,n+3)$, $\det T=1$ and $T$ preserves the signature of the first coordinate of any $x\in \mathbb{R}^{n+4}_1$ (i.e, it preserves the time direction).

Let $y:M^2\rightarrow \SSS^{n+2}$ be a conformal immersion from a Riemann surface $M$. Let $U\subset M$ be an open subset. A local
lift of $y$ is a map $Y:U\rightarrow \mathcal{C}_+^{n+3}$ such
that $\pi\circ Y=y$. Two different local lifts differ by a scaling,
so with conformal induced metrics. Here we call $y$ a conformal immersion, if $\langle Y_{z},Y_{z}\rangle=0$ and $\langle
Y_{z},Y_{\bar{z}}\rangle >0$ for any local lift $Y$ and any complex
coordinate $z$ on $M$. Then there is a decomposition $M\times
\mathbb{R}^{n+4}_{1}=V\oplus V^{\perp}$, where
\[
V={\rm Span}\{Y,{\rm Re}Y_{z},{\rm Im}Y_{z},Y_{z\bar{z}}\}
\]
is a Lorentzian rank-4 sub-bundle independent of the choice of $Y$
and $z$. Their
complexifications are denoted separately as $V_{\mathbb{C}}$ and
$V^{\perp}_{\mathbb{C}}$.

Fix a local coordinate $z$. There is a local lift $Y$ satisfying
$|{\rm d}Y|^2=|{\rm d}z|^2$,  called the canonical lift (with respect
to $z$). Choose a frame $\{Y,Y_{z},Y_{\bar{z}},N\}$ of
$V_{\mathbb{C}}$, where $N\in\Gamma(V)$ is uniquely determined by
\begin{equation}\label{eq-N}
\langle N,Y_{z}\rangle=\langle N,Y_{\bar{z}}\rangle=\langle
N,N\rangle=0,\langle N,Y\rangle=-1.
\end{equation}

Now we define \emph{the conformal Gauss map} of $y$ as follow. See
also \cite{Bryant1984,BPP,Ejiri1988,Ma}.

\begin{definition}
For a conformally immersed surface $y:M\to \SSS^{n+2}$ with canonical
lift $Y$ (with respect to a local coordinate $z$), we define
\[
G:=Y\wedge Y_{u}\wedge Y_{v}\wedge N=-2i\cdot Y\wedge Y_{z}\wedge
Y_{\bar{z}} \wedge N,\ z=u+iv,
\]
where $N\equiv 2Y_{z\bar{z}}(\!\!\mod Y)$ is the frame vector
determined in \eqref{eq-N}. It is direct to see that $G$ is well defined. We call $G:M\rightarrow
Gr_{3,1}(\mathbb{R}^{n+4}_{1})$ \emph{the conformal Gauss map} of
$y$.
\end{definition}

Given frames as above, and noting that $Y_{zz}$ is orthogonal to
$Y$, $Y_{z}$ and $Y_{\bar{z}}$, there exists a complex function $s$
and a section $\kappa\in \Gamma(V_{\mathbb{C}}^{\perp})$ such that
\[
Y_{zz}=-\frac{s}{2}Y+\kappa.
\]
This defines two basic invariants $\kappa$ and $s$ depending on
coordinates $z$, \emph{the conformal Hopf differential} and
\emph{the Schwarzian} of $y$ (for more discussion, see \cite{BPP,Ma}). Let $D$
denote the normal connection and $\psi\in
\Gamma(V_{\mathbb{C}}^{\perp})$ any section of the normal bundle.
The structure equations can be given as follows:
\[
\left\{\begin {array}{lllll}
Y_{zz}=-\frac{s}{2}Y+\kappa,\\
Y_{z\bar{z}}=-\langle \kappa,\bar\kappa\rangle Y+\frac{1}{2}N,\\
N_{z}=-2\langle \kappa,\bar\kappa\rangle Y_{z}-sY_{\bar{z}}+2D_{\bar{z}}\kappa,\\
\psi_{z}=D_{z}\psi+2\langle \psi,D_{\bar{z}}\kappa\rangle Y-2\langle
\psi,\kappa\rangle Y_{\bar{z}}.
\end {array}\right.
\]
The conformal Gauss, Codazzi and Ricci equations as integrable
conditions are:
\begin{equation}\label{eq-integ}
\left\{\begin {array}{lllll} \frac{1}{2}s_{\bar{z}}=3\langle
\kappa,D_z\bar\kappa\rangle +\langle D_z\kappa,\bar\kappa\rangle,\\
{\rm Im}(D_{\bar{z}}D_{\bar{z}}\kappa+\frac{\bar{s}}{2}\kappa)=0,\\
R^{D}_{\bar{z}z}=D_{\bar{z}}D_{z}\psi-D_{z}D_{\bar{z}}\psi =
2\langle \psi,\kappa\rangle\bar{\kappa}- 2\langle
\psi,\bar{\kappa}\rangle\kappa.
\end {array}\right.
\end{equation}

The conformal
Hopf differential plays an important role in the study of Willmore
surfaces. To see this, we first give the transformation formula of
$\kappa$. For another complex coordinate $w$,
$Y_1=Y\cdot|\frac{\dd w}{\dd z}|$ is the canonical lift with respect to $w$.
So the corresponding Hopf differential $\kappa_1$ with respect to
$(Y_1,w)$ is
\begin{equation}\label{eq-kappa-trans}
\kappa_1=\kappa\cdot\left(\frac{\dd z}{\dd w}\right)^2/\lvert\frac{\dd z}{\dd w}\rvert.
\end{equation}

 Direct computation using the structure equations above shows that $G$ induces a conformal-invariant
metric
\[
g:=\frac{1}{4}\langle {\rm d}G,{\rm d}G\rangle=\langle
\kappa,\bar{\kappa}\rangle|\dd z|^{2}
\]
on M. Note this metric degenerates at umibilic points of $y$.  We define the Willmore
functional and Willmore surfaces by use of this metric.

\begin{definition} \emph{The Willmore functional} of $y$ is
defined as the area of M with respect to the metric above:
\[
W(y):=2i\int_{M}\langle \kappa,\bar{\kappa}\rangle \dd z\wedge
\dd \bar{z}.
\]
An immersed surface $y:M\rightarrow \SSS^{n+2}$ is called a
\emph{Willmore surface} if it is a critical surface of the Willmore
functional with respect to any variation of the map $y:M\rightarrow
 \SSS^{n+2}$.
\end{definition}
It is direct to verify that
$W(y)$ is well-defined from the formula \eqref{eq-kappa-trans}. Willmore surfaces can be characterized as
follows
\cite{Bryant1984,BPP,Ejiri1988,Wang1998}:

\begin{theorem}\label{thm-willmore} For a conformal immersion $y:M\rightarrow  \SSS^{n+2}$, the following three conditions
are equivalent:
\begin{enumerate}
\item The immersion $y$ is Willmore.
\item The conformal Gauss map $G$ is a harmonic map into
$G_{3,1}(\mathbb{R}^{n+3}_{1})$.
\item The conformal Hopf differential $\kappa$ of $y$ satisfies the
following Willmore condition, which is stronger than the conformal
Codazzi equation \eqref{eq-integ}:
\[
D_{\bar{z}}D_{\bar{z}}\kappa+\frac{\bar{s}}{2}\kappa=0.
\]
\end{enumerate}
\end{theorem}

In the seminal paper \cite{Bryant1984}, Bryant showed that every Willmore surface $Y$ in $\SSS^3$ admits a dual Willmore surface $\hat{Y}$, i.e., another map $\hat{Y}$, which may have
branch points or degenerate to a point, but, if immersed,  has the same complex coordinate and the same conformal Gauss map as $Y$. This duality theorem, however, does not hold in general when the codimension is bigger than $1$ (\cite{Ejiri1988}, \cite{BPP}, \cite{Ma2006}). To characterize  Willmore surfaces with dual surfaces, in \cite{Ejiri1988} Ejiri  introduced the  notion  of {\em S-Willmore surfaces}. Here we define it slightly differently to include all Willmore surfaces with dual surfaces:
\begin{definition}
A Willmore immersion $y:M^2\rightarrow \SSS^{n+2}$ is called an S-Willlmore surface if its conformal Hopf differential satisfies
\[
D_{\bar{z}}\kappa || \kappa,
\]
i.e. there exists some function $\mu$ on $M$ such that $D_{\bar{z}}\kappa+\frac{\mu}{2}\kappa=0$.
\end{definition}
A basic result of \cite{Ejiri1988} states that a Willmore surface admits a dual surface if and only if it is S-Willmore. Moreover the dual surface is also Willmore, when it is non-degenerate.

\begin{example}

1. It is well known that minimal surfaces in Riemannian space forms are Willmore surfaces (see \cite{Bryant1984,Kusner1989} for example). These surfaces give the basic examples of Willmore surfaces. Moreover, they are, in any codimension, S-Willmore surfaces, i.e., Willmore surfaces with a dual surface, see \cite{Ejiri1988,Ma}.

2. Using the Hopf bundle, Pinkall \cite{Pinkall1985} obtained a family of non-minimal Willmore surfaces in $\SSS^3$  via the elastic curves.

\end{example}

\subsection{Harmonic maps into $SO^+(1,4)/\left(SO^+(1,1)\times SO(3)\right)$ related to
Willmore surfaces}  \label{section-anotherharmonicmap}

In the classic paper \cite{Helein}, H\'{e}lein showed that there exists another family of flat connections associated with an umbilic free Willmore surface in $\SSS^3$, besides the one related to the conformal Gauss map. H\'elein's connections yield many ``roughly harmonic" maps $Y\wedge \hat{Y}$, that take values in
$SO^+(1,4)/\left(SO^+(1,1)\times SO(3)\right)$. Here $\hat{Y}$ is an arbitrary lightlike vector other than $Y$ in the mean curvature sphere $V$ of $Y$. Moreover, he found that if $\hat{Y}$ is chosen suitably (which yields a Riccati equation), the roughly harmonic map $Y\wedge \hat{Y}$ will be truly harmonic \cite{Helein}. A special choice is to set $\hat{Y}$ to be  the dual surface of $Y$ (\cite{Helein}, \cite{Helein2}). These results are generalized for Willmore surfaces in $\SSS^{n+2}$ in \cite{Xia-Shen}.

In a different approach Ma \cite{Ma2006} proved that a Willmore surface in $\SSS^{n+2}$ locally  always admits an adjoint transform (which in general may be non-unique). This is the generalization of the duality theorem of Willmore surfaces in $\SSS^3$. Furthermore, he found that a Willmore surface together with an adjoint transform, derives a new kind of harmonic map into $SO^+(1,n+3)/\left(SO^+(1,1)\times SO(n+2)\right)$, which turns out to be one of the harmonic maps found by H\'{e}lein \cite{Helein} and Qiaoling Xia, Yibing Shen \cite{Xia-Shen}.

To avoid burdening the reader who may be primarily concerned with the $\SSS^3$ case with unnecessary information, we will restrict ourselves, in this subsection and the sections
immediately following,
to Willmore surfaces in $\SSS^3$. The general case of $\SSS^{n+2}$  includes more possibilities, which we discuss in Section \ref{section7}.

Let $y:U\rightarrow \SSS^{3}$ be an umbilic free Willmore surface  with canonical lift $Y$ with respect to $z$ as above. We introduce $\hat{Y}$ as
\begin{equation}\label{eq-hat-Y-def}
\hat{Y}=N+ \bar\mu Y_{z}+ \mu Y_{\bar{z}}+ \frac{1}{2}|\mu|^2Y.
\end{equation}
with $\mu \dd z=2\langle\hat{Y},Y_z\rangle \dd z$ a complex connection 1-form.
Direct computation yields
\[
\hat{Y}_{z}=\frac{\mu}{2} \hat{Y}+\theta \left(Y_{\bar{z}}+\frac{\bar\mu}{2}Y\right)+\rho  \left(Y_{z}+\frac{\mu}{2}Y\right)+2\zeta
\]
with
\[
\theta:=
\mu_z- \frac{\mu^2}{2}-s,\quad  \rho:=\bar\mu_{z}-2\langle \kappa,\bar\kappa \rangle,
\quad \zeta:=D_{\bar{z}}\kappa+\frac{\bar{\mu}}{2}\kappa.
\]
Then
$\hat{Y}$ is the dual surface of $Y$ if and only if $D_{\bar{z}}\kappa+\frac{\bar{\mu}}{2}\kappa=0$ (\cite{Bryant1984}, \cite{Ejiri1988}, \cite{Ma}, \cite{Ma2006}). Note now the Willmore equation is equivalent to the Riccati equation
\begin{equation}\label{eq-h}
\mu_z-\frac{\mu^2}{2}-s=0.
\end{equation}

\begin{theorem}\cite{Helein}, \cite{Xia-Shen}, \cite{Ma2006}\ {\em (Harmonicity of another map)}
Let $Y$ be an  umbilic free Willmore surface in $\SSS^3$ with $\hat{Y}$ its dual surface. Set
\[
\begin{array}{cccc}
f_h: & U &\rightarrow & SO^{+}(1,n+3)/\left(SO^+(1,1)\times SO(n+2)\right) \\
    \ & p\in U&\mapsto & Y(p)\wedge \hat{Y}(p). \\
\end{array}
\]
Then  $f_h$ is a conformally harmonic map.
\end{theorem}

At umbilic points it is possible that there exists a limit of $\mu$ such that \eqref{eq-h} holds. Due to the following lemma, the harmonic map $f_h$ has no definition when $\mu$ tends to $\infty$.

\begin{lemma} \label{umbiliclemma}
\cite{DoWa1} At the umbilic points of $Y$,  the limit of $\mu$ goes to a finite number or infinity. When $\mu$ goes to infinity, $[\hat{Y}]$ tends to $[Y]$, and at the point in question
we have $[\hat{Y}]=[Y]$.
\end{lemma}

In order to use the machinery of loop groups, we need to examine the structure of the Maurer-Cartan form
of a frame for $Y \wedge \hat Y$:
\begin{proposition}
Let  $f_h=Y\wedge\hat{Y}$ be a harmonic map, where $Y$ and $\hat Y$ are a
Willmore surface and its dual, as above. Chose a frame
\[
F=\left(\frac{1}{\sqrt{2}}(Y+\hat{Y}),\frac{1}{\sqrt{2}}(-Y+\hat{Y}),P_1,P_2,\psi\right):U\rightarrow SO^+(1,4)
\]
with $Y_{z}+\frac{\mu}{2} Y=\frac{1}{2}(P_1-iP_2)$, and $\psi$ a unit vector in the normal bundle $V^{\perp}$. Set $\kappa=k\psi.$
Then the Maurer-Cartan form $\alpha=F^{-1}\dd F=\alpha^\prime+\alpha^{\prime \prime}$ of $F$ is
\[
\alpha^\prime =\left(
                   \begin{array}{cc}
                     A_1 & B_1 \\
                     -B_1^tI_{1,1} & A_2 \\
                   \end{array}
                 \right)\dd z,
\]
with
\[
A_1=\left(
                      \begin{array}{cc}
                        0 & \frac{\mu}{2} \\
                        \frac{\mu}{2} & 0 \\
                      \end{array}
                    \right),\
 B_1=\left(
      \begin{array}{ccccccc}
                       \frac{1+\rho}{2\sqrt{2}} &  \frac{-i-i\rho}{2\sqrt{2}}  & 0\\
                        \frac{1-\rho}{2\sqrt{2}} &  \frac{-i+i\rho}{2\sqrt{2}} & 0 \\
      \end{array}
    \right)=\left(
              \begin{array}{c}
                b_1^t \\
                b_2^t \\
              \end{array}
            \right).
\]
So
 \[
 B_1B_1^t=0.
\]
\end{proposition}

It is straightforward to see that this last condition  on $B_1$ is independent of the choice of frame $F$ for the harmonic map $f_h$.
Conversely, this condition is also sufficient to characterize Willmore surfaces:
 \begin{theorem}\cite{Helein}, \cite{Helein2}, \cite{Xia-Shen}, \cite{Ma2006}.
Let  $f:M\rightarrow SO^+(1,4)/(SO^+(1,1)\times SO(3))$ be a non-constant harmonic map satisfying $B_1B_1^t=0$. Then $Y$ and $\hat{Y}$ are a pair of dual (possibly degenerate) Willmore surfaces. Moreover, set \[
B_1=(b_1\ b_2)^t \hbox{ with } b_1,b_2\in\mathbb{C}^3.
\]
Then $Y$ is immersed at the points $(b_1^t+b_2^t)(\bar{b}_1+\bar{b}_2)>0$ and $\hat{Y}$ is immersed  at the points  $(b_1^t-b_2^t)(\bar{b}_1-\bar{b}_2)>0$.
\end{theorem}

Note that $Y$ or $\hat{Y}$ may degenerate to a point, and in this case the dual
($\hat Y$ or $Y$) is M\"obius equivalent to a minimal surface in $\mathbb{R}^3$.

Since $B_1B_1^t=0$ serves as some isotropic condition, we define:
 \begin{definition} \label{isotropicdef}
Let $f:M\rightarrow SO^+(1,4)/(SO^+(1,1)\times SO(3))$ be a non-constant harmonic map. Then $f$ is called an \emph{isotropic} harmonic map if the Maurer-Cartan form of any frame of $f$, with the above notation, satisfies $B_1B_1^t=0$.
\end{definition}
This characterization of Willmore surfaces in terms of isotropic harmonic maps essentially
follows from the work of H\'elein \cite{Helein,Helein2}, although the name ``isotropic"
is not used there.

 \section{Isotropic harmonic maps into $SO^+(1,4)/(SO^+(1,1)\times SO(3))$}
\label{section3}

\subsection{Harmonic maps into a Symmetric space}

Let $N=G/K$ be a symmetric space with involution $\sigma: G\rightarrow G$ such that $G^{\sigma}\supset K\supset(G^{\sigma})_0$. Let $\mathfrak{g}$ and $\mathfrak{k}$ denote the Lie algebras of $G$ and $K$ respectively. The Cartan decomposition shows that
\[
\mathfrak{g}=\mathfrak{k}\oplus\mathfrak{p}, \quad
[\mathfrak{k},\mathfrak{k}]\subset\mathfrak{k}, \quad
[\mathfrak{k},
\mathfrak{p}]\subset\mathfrak{p},
\quad [\mathfrak{p},\mathfrak{p}]\subset\mathfrak{k}.
\]
 Denote $\pi:G\rightarrow G/K$ the projection of $G$ into $G/K$.

Let $f:M\rightarrow G/K$ be a conformal harmonic map from a connected, oriented, closed surface $M$. Let $U\subset M$ be an open connected subset.
Then there exists a frame $F: U\rightarrow G$ such that $f=\pi\circ F$. So we have the Maurer-Cartan form and  Maurer-Cartan equation
\[
F^{-1}\dd F= \alpha, \ \dd \alpha+\frac{1}{2}[\alpha\wedge\alpha]=0.
\]
Decomposing these with respect to $\mathfrak{g}=\mathfrak{k}\oplus\mathfrak{p}$ amounts to:
\[
\alpha=\alpha_0+\alpha_1, \quad  \alpha_0\in \Gamma(\mathfrak{k}\otimes T^*M), \quad
\alpha_1\in \Gamma(\mathfrak{p}\otimes T^*M),
\]
\[
\left\{\begin{array}{ll}
 & \dd \alpha_0+\frac{1}{2}[\alpha_0\wedge\alpha_0]+\frac{1}{2}[\alpha_1\wedge\alpha_1]=0.\\
& \dd \alpha_1+[\alpha_0\wedge\alpha_1]=0.
\end{array}\right.
\]
Decomposing $\alpha_1$ further into the $(1,0)-$part $\alpha_1^\prime$ and the $(0,1)-$part $\alpha_1^{\prime \prime}$, we then set
 \[
\alpha_{\lambda}=\lambda^{-1}\alpha_{1}^\prime+\alpha_0+\lambda\alpha_{1}^{\prime \prime},  \quad \lambda\in \SSS^1.
\]
We have the famous characterization in terms of one-parameter families:
\begin{lemma} $($\cite{DPW}$)$ The map  $f:M\rightarrow G/K$ is harmonic if and only if
\[
 \dd \alpha_{\lambda}+\frac{1}{2}[\alpha_{\lambda}\wedge\alpha_{\lambda}]=0\ \ \hbox{for all}\ \lambda \in \SSS^1.
\]
\end{lemma}

\begin{definition}The frame $F(z,\lambda)$,  solving from the equation
\[
\dd F(z,\lambda)= F(z, \lambda) \, \alpha_{\lambda}
\]
with the initial condition $F(0,\lambda)=F(0)$,
is called the {\em extended frame }
 of the harmonic map $f$. Note that it satisfies $F(z,1)=F(z)$.
 \end{definition}

\subsection{The DPW construction of harmonic maps}

\subsubsection{Two decomposition theorems}

We denote by $SO^+(1,n+3)$ the connected component of the identity of the linear isometry group of $\mathbb{R}^{n+4}_1$, with the metric introduced in Section \ref{section2}.
 Then
\[
\mathfrak{s}o(1,n+3)=\mathfrak{g}=\{X\in \mathfrak{g}l(n+4,\mathbb{R})|X^tI_{1,n+3}+I_{1,n+3}X=0\}.
\]
Consider the involution
 \[
\begin{array}{ll}
\sigma:  SO^+(1,n+3)& \rightarrow SO^+(1,n+3)\\
 \ \ \ \ \ \ \ A&\mapsto DAD^{-1},
\end{array}
\quad \quad \hbox{where} \quad
D=\left(
         \begin{array}{ccccc}
             -I_{2} & 0 \\
            0 & I_{n+2} \\
         \end{array}
       \right).
\]
We have $SO^+(1,n+3)^{\sigma}\supset SO^+(1,1)\times SO(n+2)= (SO^+(1,n+3)^{\sigma})_0$. We also have
\[
\mathfrak{g}=
\left\{\left(
                   \begin{array}{cc}
                     A_1 & B_1 \\
                     -B_1^tI_{1,1} & A_2 \\
                   \end{array}
                 \right)
 |A_1^tI_{1,1}+I_{1,1}A_1=0,A_2+A_2^t=0\right\}=\mathfrak{k}\oplus\mathfrak{p},
\]
with
\[
\mathfrak{k}=\left\{\left(
                   \begin{array}{cc}
                     A_1 &0 \\
                     0 & A_2 \\
                   \end{array}
                 \right)
 |A_1^tI_{1,1}+I_{1,1}A_1 =0, A_2+A_2^t=0\right\}, \]\[
\  \mathfrak{p}=\left\{\left(
                   \begin{array}{cc}
                   0 & B_1 \\
                     -B_1^tI_{1,1} & 0 \\
                   \end{array}
                 \right)
\right\}.
\]
Let $G^{\mathbb{C}}=SO^+(1,n+3,\mathbb{C}) := \{X \in SL(n+4,\C) ~|~ X^t I_{1,n+3} X =I_{1,n+3}\}$,
which has Lie algebra $\mathfrak{so}(1,n+3,\mathbb{C})$. Extend $\sigma$ to an inner involution of $SO^+(1,n+3,\mathbb{C})$  with fixed point group $K^{\mathbb{C}}=S(O^+(1,1,\mathbb{C})\times O(n+2,\C))$.

Let $\Lambda G^{\mathbb{C}}_{\sigma}$ denote the group of loops in $G^C =SO^+(1,n+3,\mathbb{C})$ with the twisting by $\sigma$. Let $\Lambda^+G^{\mathbb{C}}_{\sigma}$  denote the subgroup of loops which extend holomorphically to the unit disk $|\lambda|\leq1$. We also
use the subgroup
\[
\Lambda_B^+ G^{\mathbb{C}}_{\sigma}:=\{\gamma\in\Lambda^+G^{\mathbb{C}}_{\sigma}~|~\gamma|_{\lambda=0}\in \mathfrak{B} \}.
\]
Here $\mathfrak{B}\subset K^{\mathbb{C}}$ is defined from the Iwasawa decomposition
\[
K^{\mathbb{C}}=K\cdot\mathfrak{B}.
\]
In this case,
\[
\mathfrak{B}=\left\{\left(
                   \begin{array}{cc}
                     \mathrm{b}_1 & 0 \\
                     0 & \mathrm{b}_2 \\
                   \end{array}
                 \right)\ |\ \mathrm{b}_1=\left(
                        \begin{array}{cc}
                          \cos\theta & i\sin\theta \\
                          i\sin\theta & \cos\theta \\
                        \end{array}
                      \right), \theta\in \frac{ \real }{2\pi{\mathbb  Z}},\ ~\hbox{and } ~\mathrm{b}_2\in \mathfrak{B}_2\right
\}.
\]
Here  $\mathfrak{B}_2$  is the solvable subgroup of $SO(n+2,\mathbb{C}).$
                      For more details, see Lemma 4 of \cite{Helein}. Then we have:
\begin{theorem}\label{thm-iwasawa} Theorem 5 of \cite{Helein}, see also \cite{Xia-Shen},  \cite{DPW}, \cite{PS}, \cite{B-R-S} (Iwasawa decomposition):
The multiplication $\Lambda G_{\sigma}\times \Lambda^+_{B}G^{\mathbb{C}}\rightarrow\Lambda G^{\mathbb{C}}_{\sigma}$ is a real analytic diffeomorphism onto the open dense subset $\Lambda G_{\sigma}\cdot \Lambda^+_{B}G^{\mathbb{C}} \subset\Lambda G^{\mathbb{C}}_{\sigma}$.
\end{theorem}

Let $\Lambda^-_*G^{\mathbb{C}}_{\sigma}$ denote the loops that extend holomorphically into $\infty$ and take values $I$ at infinity.
\begin{theorem} \label{thm-birkhoff}Theorem 7 of \cite{Helein}, see also \cite{Xia-Shen}, \cite{DPW}, \cite{PS}, \cite{B-R-S}  (Birkhoff decomposition):
The multiplication $\Lambda^-_* G^{\mathbb{C}}_{\sigma}\times \Lambda^+G^{\mathbb{C}}\rightarrow\Lambda G^{\mathbb{C}}_{\sigma}$ is a real analytic diffeomorphism onto the open subset $ \Lambda^-_* G^{\mathbb{C}}_{\sigma}\cdot \Lambda^+G^{\mathbb{C}}$ (the big cell) of $\Lambda G^{\mathbb{C}}_{\sigma}$.

\end{theorem}

\subsubsection{The DPW construction and Wu's formula}

Here we recall the DPW construction for harmonic maps. Let $\mathbb{D}\subset\mathbb{C}$ be a disk or $\mathbb{C}$ itself, with complex coordinate $z$.
\begin{theorem} \label{thm-DPW} \cite{DPW}
\begin{enumerate}
\item Let $f:\mathbb{D}\rightarrow G/K$ be a harmonic map with an extended frame $F(z,\bar{z},\lambda)\in \Lambda G_{\sigma}$ and $F(0,0,\lambda)=I$. Then there exists a Birkhoff decomposition
\[
F_-(z,\lambda)=F(z,\bar{z},\lambda)F_+(z,\bar{z},\lambda),~ \hbox{ with }~ F_+\in\Lambda^+G^{\mathbb{C}}_{\sigma},
\]
such that $F_-(z,\lambda):\mathbb{D} \rightarrow\Lambda^-_*G^{\mathbb{C}}_{\sigma}$ is meromorphic. Moreover, the Maurer-Cartan form of $F_-$ is of the form
\[
\eta=F_-^{-1}\dd F_-=\lambda^{-1}\eta_{-1}(z)\dd z,
\]
with $\eta_{-1}$ independent of $\lambda$. The $1$-form $\eta$ is called the normalized potential of $f$.
\item Let $\eta$ be a $\lambda^{-1}\cdot\mathfrak{p}-$valued meromorphic 1-form on $\mathbb{D}$. Let $F_-(z,\lambda)$ be a solution to $F_-^{-1}\dd F_-=\eta$, $F_-(0,\lambda)=I$. Then on an open subset $\mathbb{D}_{\mathfrak{I}}$ of $\mathbb{D}$ one has
\[
F_-(0,\lambda)=\tilde{F}(z,\bar{z},\lambda)\cdot \tilde{F}^+(z,\bar{z},\lambda),\ \hbox{ with }\ \tilde{F}\in\Lambda G_{\sigma},\ \tilde{F}\in\Lambda ^+_{B} G^{\mathbb{C}}_{\sigma}.
\]
This way, one obtains an extended frame $\tilde{F}(z,\bar{z},\lambda)$ of some harmonic map from  $\mathbb{D}_{\mathfrak{I}}$  to $G/K$ with $\tilde{F}(0,\lambda)=I$. Moreover, all harmonic maps can be obtained in this way, since these two procedures are inverse to each other if the normalization at some based point is used.
\end{enumerate}
\end{theorem}
The normalized potential can be determined in the following way. Let  $f$ and $F$ be as above. Let $\alpha_{\lambda}=F^{-1}\dd F$. Let $\delta_1$ and $\delta_0$ denote the sum of the holomorphic terms of $z$ around $z=0$ in the Taylor expansion of $\alpha_1^\prime (\frac{\partial}{\partial z})$ and  $\alpha_0^\prime (\frac{\partial}{\partial z})$.
 \begin{theorem} \label{thm-wu} \cite{Wu} (Wu's formula) We retain the  notations of Theorem \ref{thm-DPW}. The the normalized potential of $f$ with respect to the base point $0$ is given by
 \[
 \eta=\lambda^{-1}\Delta_0 \delta_1\Delta_0 ^{-1} \dd z,
\]
where $\Delta_0 :\mathbb{D}\rightarrow G^{\mathbb{C}}$ is the solution to $\Delta_0 ^{-1}\dd \Delta_0 =\delta_0 \dd z$, $\Delta_0(0)=I$.
\end{theorem}

For many applications, normalized potentials are too specific.
Another type of \emph{holomorphic} potential was also
introduced in \cite{DPW}:
\begin{theorem} \label{thm-DPW2} \cite{DPW}
  We retain the notations of $f$ and $F(z,\bar{z},\lambda)$ in Theorem \ref{thm-DPW}. Then there exists some $V+:\mathbb{D}\rightarrow \Lambda^+G^{\mathbb{C}}_{\sigma}$ such that
	\[
	C(z,\lambda)=F(z,\bar{z},\lambda)V_+(z,\bar{z},\lambda)
	\]
is holomorphic in $z$ and in $\lambda\in \mathbb{C}^*$. Moreover, the Maurer-Cartan form $\Xi=C^{-1} \dd C$ is a holomorphic $1-$form on $
\mathbb{D}$ with $\lambda\eta$ holomorphic in $\lambda$ for all $\lambda\in\mathbb{C}$. The $1$-form $\Xi$ is called a holomorphic potential of $f$.

Conversely, let $\Xi$ be a $\Lambda\mathfrak{g}^{\mathbb{C}}_{\sigma}-$valued holomorphic 1-form on $\mathbb{D}$ such that $\lambda\Xi$ is holomorphic in $\lambda$ for all $\lambda\in \mathbb{C}$. Let $C$ be a solution to $C^{-1} \dd C=\Xi$, $C(0,\lambda)=I$. Then on an open subset $\mathbb{D}_{\mathfrak{I}}$ of $\mathbb{D}$, one obtains
\[
C(z,\lambda)=\hat{F}(z,\bar{z},\lambda)\cdot \hat{V}_+(z,\bar{z},\lambda),~ \hbox{ with }~ \tilde{F}\in\Lambda G_{\sigma},\ \hat{V}_+  \in\Lambda ^+_{B} G^{\mathbb{C}}_{\sigma}.
\]
Hence, one obtains an extended frame $\hat{F}(z,\bar{z},\lambda)$ of some harmonic map from  $\mathbb{D}_{\mathfrak{I}}$  to $G/K$ with $\hat{F}(0,\lambda)=I$. Moreover, all harmonic maps can be obtained in this way.
\end{theorem}
Note that there exist many different holomorphic potentials for a harmonic map.

\subsection{Potentials of isotropic harmonic maps}
  Let  $\mathbb{D}$ denote the unit disk of $\mathbb{C}$ or $\mathbb{C}$ itself.
  \begin{theorem}\label{thm-np of HA} \cite{Helein}, \cite{Helein2}
\begin{enumerate}
\item  Let $f: \mathbb{D}\rightarrow SO^+(1,4)/(SO^+(1,1)\times SO(3))$
 be an isotropic harmonic map with complex coordinate $z$. Then its normalized potential satisfies
 \[
       \eta=\lambda^{-1}\left(
                   \begin{array}{cc}
                    0 & \hat{B}_1 \\
                     -\hat{B}_1^tI_{1,1} & 0\\
                   \end{array}
                 \right)\dd z, \ \hbox{ with }\ \hat{B}_1\hat{B}_1^t=0.
\]
Conversely, let $f$ be the harmonic map derived from a normalized potential $\eta$ satisfying the above condition. Then $f=Y\wedge\hat{Y}$ is an isotropic harmonic map associated with the dual Willmore surfaces $Y$ and $\hat{Y}$.
\item
 Let $f: \mathbb{D}\rightarrow SO^+(1,4)/(SO^+(1,1)\times SO(3))$
 be an isotropic harmonic map with complex coordinate $z$. Then any holomorphic potential of $f$ satisfies
 \[
\Xi=\sum_{j=-1}^\infty \lambda^{j}\xi_{j}\dd z, \ \hbox{ with }\ \xi_{-1}=\left(
                   \begin{array}{cc}
                    0 & \tilde{B}_1 \\
                     -\tilde{B}_1^tI_{1,1} & 0\\
                   \end{array}
                 \right) \ \hbox{ and }\ \tilde{B}_1\tilde{B}_1^t=0.
\]
Conversely, let $f$ be the harmonic map derived from a holomorphic potential $\Xi$ satisfying the condition above. Then $f=Y\wedge\hat{Y}$ is an isotropic harmonic map associated with the dual Willmore surfaces $Y$ and $\hat{Y}$.
\end{enumerate}
\end{theorem}

The proof comes directly from the decompositions $F=F_-\cdot F+$ and $F=C\cdot V_+$, and the fact that conjugation by some $T\in SO^+(1,1,\mathbb{C})\times SO(3,\mathbb{C})$ does not change the isotropic condition $B_1B^t_1=0$.\\

In \cite{Helein}, there is an interesting description of Willmore surfaces
M\"obius equivalent to minimal surfaces in space forms. Here we restate it as:
\begin{theorem}\label{thm-minimal}(\cite{Helein}) Let $f_h=Y\wedge \hat{Y}$ be a non-constant isotropic harmonic map.
\begin{enumerate}
\item The map $[Y]$
 is M\"obius equivalent to  a minimal surface in $\mathbb{R}^{3}$ if $\hat{Y}$ reduces to a point. In this case
\[
B_1= \left(
          \begin{array}{cc}
            b_1 &
            b_1  \\
          \end{array}
        \right)^t.
\]
\item
The map $[Y]$ is M\"obius equivalent to  a minimal surface in $\SSS^{3}$ if $f_h$ reduces to a harmonic map into $SO(4)/SO(3)$. In this case
\[
B_1= \left(
          \begin{array}{cc}
            0 &
            b_1 \\
          \end{array}
        \right)^t.
\]
\item
The map $[Y]$ is M\"obius equivalent to  a minimal surface in $H^{3}$  if $f_h$ reduces to a harmonic map into $SO^+(1,3)/SO^+(1,2)$. In this case
\[
B_1= \left(
          \begin{array}{cc}
            b_1 &
            0\\
          \end{array}
        \right)^t .
\]
\end{enumerate}
Here $b_1\in\mathbb{C}^{4}$ and $b_1^tb_1=0$.

The converse of the above results also hold. That is, if $B_1$ is (up to conjugation) of the form stated above, then $[Y]$ is M\"obius equivalent to  the corresponding minimal surface where it is an immersion.
\end{theorem}

\subsection{Examples}
By implementing the Iwasawa decomposition numerically, one can compute solutions
and plot the images of Willmore surfaces with the aid of a computer.
Here are some simple examples, with images shown at Figure \ref{figureexamples}.

\begin{example} \label{ex_cliff}
Let
\begin{equation*} \eta=\lambda^{-1}\left(
                   \begin{array}{cc}
                    0 & \hat{B}_1 \\
                     -\hat{B}_1^tI_{1,1} & 0\\
                   \end{array}
                 \right)\dd z, \ \hbox{ with }\ \hat{B}_1=\left(
                                                         \begin{array}{ccc}
                                                           b_1 & b_2 \\
                                                         \end{array}
                                                       \right)^t .
\end{equation*}
It is shown in \cite{Helein},
that if one chooses
\[b_1=0,\ b_2^t =\frac{\sqrt{2}}{4}\left(
                             \begin{array}{ccccc} 1-\frac{z^2}{8} & -i(1+\frac{z^2}{8})& \frac{\sqrt{2}z}{2} \\
                             \end{array}
                           \right)
                           \]
                           one will obtain the Clifford torus in $\SSS^3$. Note that $b_2$ is exactly the Weierstrass-representation data of the Enneper surface.
                         \end{example}

\begin{figure}[ht]
\centering
$
\begin{array}{ccc}
\includegraphics[height=30mm]{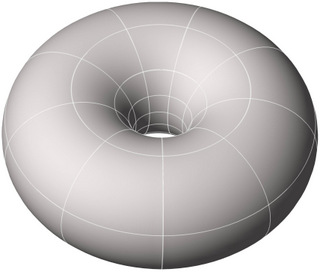} \,\, & \,\,
\includegraphics[height=30mm]{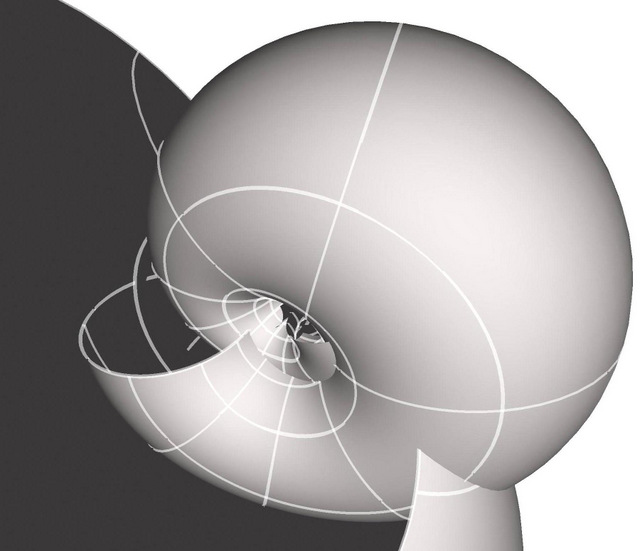} \,\, & \,\,
\includegraphics[height=30mm]{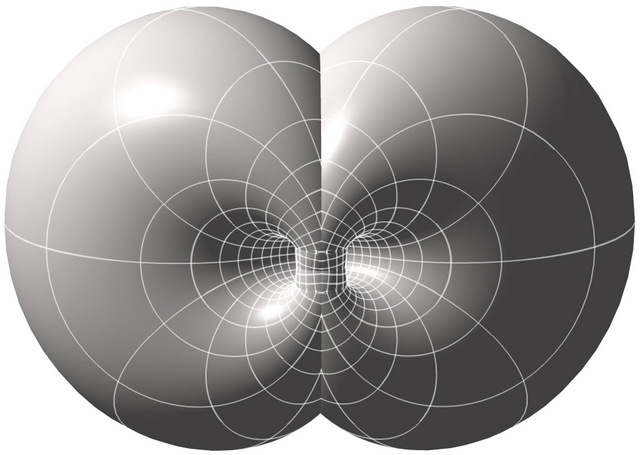}
\end{array}
$
\caption{Willmore surfaces computed with a numerical implementation of
DPW. Left: Example \ref{ex_cliff}.
Middle: Example \ref{ex_gencliff}. Right: Example \ref{ex_invcliff}.}
\label{figureexamples}
\end{figure}

                         \begin{example} \label{ex_gencliff}
												If we choose  \[b_1=\frac{i}{4}\left(
                             \begin{array}{ccccc} 1-\frac{z^2}{8} & -i(1+\frac{z^2}{8})& \frac{\sqrt{2}z}{2} \\
                             \end{array}
                           \right),\ b_2^t =\frac{\sqrt{3}}{4}\left(
                             \begin{array}{ccccc} 1-\frac{z^2}{8} & -i(1+\frac{z^2}{8})& \frac{\sqrt{2}z}{2} \\
                             \end{array}
                           \right)
                           \]
 we obtain the second image in Figure \ref{figureexamples}. Note that this Willmore surface is not M\"obius equivalent to  a minimal surface in any space form, by Theorem \ref{thm-minimal}.\end{example}

                         \begin{example} \label{ex_invcliff}
    Replacing $z$ with $1/z$ in the Clifford torus potential:
\[b_1=0,\ b_2 = \left(
                             \begin{array}{ccccc} 1-\frac{1}{z^2} & -i(1+\frac{1}{z^2})& \frac{1}{z} \\
                             \end{array}
                           \right),
                           \]
and integrating with initial condition $F(1)=I$, we  obtain the third
image in the figure. This Willmore surface is M\"obius equivalent to  a minimal surface in $\SSS^3$ by Theorem \ref{thm-minimal}.
\end{example}

\section{Bj\"{o}rling's Problem for Willmore surfaces in $\SSS^3$}
\label{section4}

We state the Bj\"{o}rling problem for Willmore surfaces in $\SSS^3$ as:
Given a sphere congruence together with two enveloping curves on an interval $\mathbb{I}$ of $\SSS^3$, does there exist a unique pair of dual Willmore surfaces such that their restrictions to the interval $\mathbb{I}$  coincide with the two enveloping curves separately and their mean curvature sphere coincides with the sphere congruence?

Concretely, we have the following result:
\begin{theorem}\label{thm-Bjorling}
Let $\psi_0=\psi_0(u):\mathbb{I}\rightarrow \SSS^4_1$ denote a non-constant real analytic sphere congruence from $\mathbb{I}$ to $\SSS^3$, with enveloping curves $[Y_0]$ and $[\hat{Y}_0]$ such that $\langle Y_0, Y_0 \rangle=\langle \hat{Y}_0, \hat{Y}_0\rangle=0$, $\langle Y_0, \hat{Y}_0\rangle=-1$, and $u$ is  the arc-parameter of $Y_0:\mathbb{I}\rightarrow \mathcal{C}_+^{n+3}$. Then there exists a unique pair of dual Willmore surfaces $y,\hat{y}:\Sigma\rightarrow \SSS^3$, with $\Sigma$ some simply connected  open subset containing $\mathbb{I}$, such that the lifts $Y,\hat{Y}$ of $y,\hat{y}$ satisfy
\[
Y|_{\mathbb{I}}=Y_0,\ \hat{Y}|_{\mathbb{I}}=\hat{Y}_0.
\]
Moreover, let $\psi:\Sigma\rightarrow \SSS^4_1$ be the conformal Gauss map of $Y$, we have
$\psi|_{\mathbb{I}}=\psi_0.$ \end{theorem}

Theorem \ref{thm-Bjorling} is a straightforward corollary of the following
\begin{theorem}\label{thm-Bjorling-BP}
We retain the assumptions and notations  in Theorem \ref{thm-Bjorling}. Choose two real analytic unit vector fields $P_1$ and $P_2$ on $\mathbb{I}$ such that
\[
Y_{0u}=P_1\mod Y_0,\ ~P_2\perp\{\psi_0,Y_0,\hat{Y}_0,P_1\}\ \hbox{ and }~\det(Y_0,\hat{Y}_0,P_1,P_2,\psi_0)=1.
\]
There exist real analytical functions
  $\mu_1=\mu_1(u)$, $\rho_1=\rho_1(u)$,  $\rho_2=\rho_2(u)$, $k_1=k_1(u)$,
	$k_2=k_2(u)$  on $\mathbb{I}$ such  that
\begin{equation}\label{eq-moving-y and dual}
\left\{\begin {array}{lllll}
Y_{0u}=-\mu_1Y_0+P_1,\\
\hat{Y}_{0u}=\mu_1\hat{Y}_0+\rho_1P_1+\rho_2P_2,\\
P_{1u}=\mu_2P_2+2k_1\psi_0+\hat{Y}_0+\rho_1 Y_0,\\
P_{2u}=-\mu_2P_1-2k_2\psi_0+\rho_2 Y_0,\\
\psi_{0u}=-2k_1P_1+2k_2P_2,\\
\end {array}\right.
\end{equation}
holds. Set $\mu=\mu_1+i\mu_2,$ $k=k_1+ik_2$ and $\rho=\rho_1+i\rho_2$.
For a real analytic function $x(u)$ on $\mathbb I$, denote its analytic extension
to a simply connected open subset containing ${\mathbb I}$ by $x(z)$.
Consider the holomorphic potential
 \[
 \Xi=\left(\lambda^{-1}\mathcal{A}_1 + \mathcal{A}_0 +\lambda\mathcal{A}_{-1} \right) \dd z,
\]
with
\beqas
\mathcal{A}_0=\bbar {cc} A_1 & 0 \\ 0 & A_2 \ebar, \quad
\mathcal{A}_1 = \bbar{cc} 0 & B_1 \\ -B_1^t I_{1,1} & 0 \ebar, \quad
\mathcal{A}_{-1}(z) = \overline{\mathcal{A}_1(\bar z)},\\
A_1(z) = \bbar{cc} 0 & \mu_1(z) \\ \mu_1(z) & 0 \ebar, \quad
A_2(z) =\bbar{ccc}   0 &  -\mu_2(z)  & -2k_1(z)  \\
       \mu_{2}(z) &0 & 2k_2(z)    \\
       2k_1(z) & -2k_2(z) & 0  \ebar  \\
B_1(z) = \frac{1}{2\sqrt{2}} \bbar{ccc}
        1+\rho(z) &  -i-i\rho(z)  & 0\\
        1-\rho(z) &  -i+i\rho(z) &0 \ebar. \\
\eeqas

By DPW, Theorem \ref{thm-DPW2}, the potential $\Xi$ provides an isotropic harmonic map, together with a unique pair of dual Willmore surfaces $y,\hat{y}:\Sigma\rightarrow \SSS^3$, with $\Sigma$ some open subset containing $\mathbb{I}$, such that the lifts $Y,\hat{Y}$ of $y,\hat{y}$ satisfy
\[
Y|_{\mathbb{I}}=Y_0,\ \hat{Y}|_{\mathbb{I}}=\hat{Y}_0.\
\]
Moreover, let $\psi:\Sigma\rightarrow \SSS^4_1$ be the conformal Gauss map of $Y$.
Then
$\psi|_{\mathbb{I}}=\psi_0.$ \end{theorem}

\begin{proof}
Set
\[
F_0=\left(\frac{Y_0+\hat{Y}_0}{\sqrt{2}},\frac{-Y_0+\hat{Y}_0}{\sqrt{2}}, P_1, P_2, \psi_0\right).
\]
Rewriting \eqref{eq-moving-y and dual}, we obtain
\[
F_0^{-1}\dd F_0=(\hat{\alpha}_1+\hat\alpha_0+ \hat{\alpha}_{-1})\dd u,
\]
with
\[\hat \alpha_0(u) = \mathcal{A}_0(u), \quad
  \hat \alpha_1(u) = \mathcal{A}_1(u),
\]
and $\mathcal{A}_j$ are as in the statement of the theorem.
Introducing $\lambda$, we set
\[
\hat{\alpha}_{\lambda}=(\lambda^{-1}\hat{\alpha}_1+\hat\alpha_0+\lambda\hat{\alpha}_{-1})\dd u.
\]
Let $F_0(u,\lambda)$ be the solution to $\dd F_0(u,\lambda)=F_0(u,\lambda)\hat{\alpha}_{\lambda}$, $F_0(u,\lambda)|_{\lambda=1}=F_0$.

Let $z=u+iv$ be the complex coordinate such that $u+i0$ parameterizes $\mathbb{I}$. As a consequence,  the holomorphic 1-form
$\Xi$ coincides with $\hat{\alpha}_{\lambda}$ when restricted to
 $\mathbb{I}$, since on $\mathbb{I}$ $z=u+i0=u$.
Assume that $\mathfrak{F}$ is the solution to
\[
\mathfrak{F}^{-1}d\mathfrak{F}=\Xi,\ \ \mathfrak{F}(u_0+i0,\lambda)=F_0(u_0+i0,\lambda) \hbox{ for some } u_0\in \mathbb{I}.
\]
Then
\[
\mathfrak{F}(z=u,\lambda)=F_0(u,\lambda),\ \ \forall u\in \mathbb{I}.
\]
Since $F_0(u,\lambda)\in\Lambda G_{\sigma}$ for all $u\in \mathbb{I}$,
$\mathfrak{F}(z)$ is in the big cell for $z$ in some open
subset $\mathbb{D}_0$ containing $\mathbb{I}$.  Performing the Iwasawa
decomposition of Theorem \ref{thm-iwasawa}, pointwise on on
$\mathbb{D}_0$,  we obtain
\[
\mathfrak{F}=\check{F}(z,\lambda)\cdot \check{F}_+(z,\lambda)
\]
with $\check{F}(z,\lambda)=\overline{\check{F}(z,\lambda)}$ on $\mathbb{D}_0$ and $ \check{F}_+(z,\lambda)\in \Lambda^+_{B}G^{\mathbb{C}}_{\sigma}$. By the initial condition we have
\[
\check{F}(z=u,\lambda)=F_0(u,\lambda),\ \ \forall u\in \mathbb{I}.
\]
By Theorem \ref{thm-DPW}, $\check{F}$ is an extended frame of some harmonic map.
                   It is straightforward to compute $\hat{B}_1\hat{B}_1^t \equiv0$. By Theorem \ref{thm-np of HA}, $\check{F}$ is an isotropic harmonic map. As a consequence,
setting $\check{F}=(e_{-1},e_0,e_1,e_2, \psi),$
then $Y=\frac{1}{\sqrt{2}}( {e}_{-1}- {e}_0)$, $ \hat{Y}=\frac{1}{\sqrt{2}}( e_{-1}+ e_0)$ and $\psi$ are the desired dual Willmore surfaces and their conformal Gauss map, which are unique and coincide, by construction, with $Y_0$, $\hat{Y}_0$ and $\psi_0$ on $\mathbb{I}$.
\end{proof}

The potential $\Xi$ defined in the above theorem is a special type of
holomorphic potential one can generally define by
taking the Maurer-Cartan
form of the extended frame $F$ for a harmonic map, restricting to some
curve in the domain, and then extending holomorphically.
  We call it the \emph{boundary potential}.

\subsection{Examples}
In the following examples we denote by
$E_0=(1,0,0,0,0)$, $\dots$ $E_4=(0,0,0,0,1)$ an orthonormal basis for $\real_1^5$,
with $\langle E_0\, , \, E_0 \rangle = -1$.
For convenience, we
write $X^\prime$ for $X_u$, and we abuse notation by dropping the subscripts on $Y_0$, $\hat Y_0$ and $\psi_0$.


\begin{example}  \label{example1}
Let us consider a Willmore surface in $\SSS^3$ containing the circle $(\cos u, \sin u, 0,0)$.  A lift is $Y=(1,\cos u, \sin u, 0,0)$. The simplest case is where the plane spanned by $Y$ and $\hat Y$ is constant:
without loss of generality we can take $\hat Y = (1/2)(1,-\cos u, -\sin u, 0,0)$.
From Equations (\ref{eq-moving-y and dual}), we have
\[
P_1 = Y^\prime + \mu_1 Y= (0,-\sin u, \cos u, 0,0) + \mu_1(1,\cos u, \sin u, 0,0).
\]
  The requirement that $\langle \hat Y, P_1 \rangle=0$ gives us:
\[ \mu_1 =0, \quad P_1 = (0,-\sin u, \cos u, 0,0).
\]
The equation $\hat Y^\prime = \mu_1 \hat Y + \rho_1 P_1 + \rho_2 P_2$ gives us
\[
\rho_1 = -1/2, \quad \rho_2 =0.
\]
The third equation from (\ref{eq-moving-y and dual}) is
\[
(0,-\cos u, -\sin u, 0,0)=P_1^\prime = \mu_2 P_2 + 2k_1 \psi + \hat Y + \rho_1 Y.
\]
Since $\psi$ and $P_2$ necessarily take values in $\textup{Span}\{E_3,E_4\}$, we conclude that
$\mu_2 = k_1 = 0$.
 The only remaining parameter for the potential is $k_2$, and this is
 determined by our choice of $\psi$, which could be any vector field taking
values in $\textup{Span}\{E_3,E_4\}$.  For example, $k_2 = 0$ corresponds to $\psi$ and $P_2$ being constant along the curve. The Willmore surface obtained is a round sphere.
More generally, we must have
\[
\psi = -\sin(\theta) E_3 + \cos(\theta) E_4, \quad
P_2 = -\cos(\theta)E_3 - \sin(\theta) E_4,
\]
 where $\theta$ is any real analytic map $\real \to \real$.
The last equation at (\ref{eq-moving-y and dual}), becomes
$\theta^\prime P_2 = \psi^\prime = 2k_2 P_2$, and so we conclude
that $k_2 = \theta^\prime/2$.
There are no further constraints, so we can say that all solutions
corresponding to the pair $Y$ and $\hat Y$ above are obtained from a choice of angle function $\theta$ with the boundary potential given by the data:
\[
(\mu,k,\rho) = (0,i\theta^\prime/2,-1/2).
\]
\end{example}

\begin{example}  \label{lawsonex}
A special case of the previous example is when $\theta^\prime$ is constant,
and for this we can write down the solution explicitly:  consider the immersion
\[
y(u, \tilde v)=(\cos u \cos \tilde v, \, \sin u \cos \tilde v, \,  \cos ru \, \, \sin \tilde v, \,  \sin ru \, \sin \tilde v),
\]
where $r$ is a non-zero real number. Note that the case that $r=\ell/m$ is rational corresponds
to Lawson's minimal tori and Klein bottles $\tau_{m,\ell}$ (see equation (7.1) of
\cite{Lawson1970}). The surfaces $\tau_{m,\ell}$ are all distinct compact genus zero
surfaces for distinct relatively prime pairs of positive integers $(m,\ell)$. They are non-orientable
if and only if $2$ divides $m$ or $\ell$.

\begin{figure}[ht]
\centering
$
\begin{array}{cccc}
\includegraphics[height=28mm]{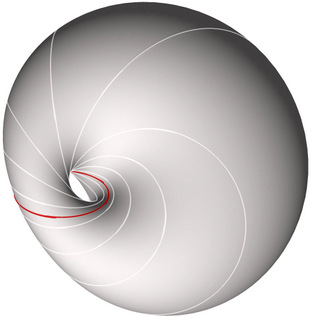} \,\, & \,\,
\includegraphics[height=28mm]{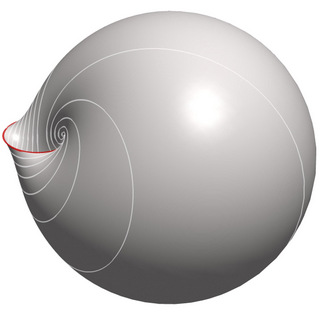} \,\, & \,\,
\includegraphics[height=28mm]{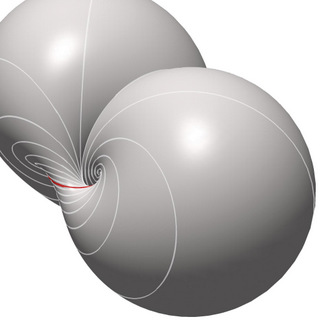} \,\, & \,\,
\includegraphics[height=28mm]{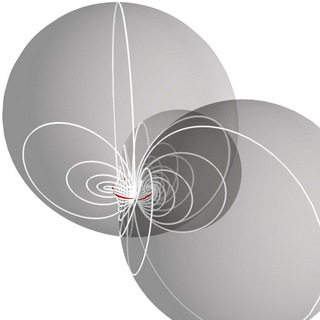}   \\
r=1 & r=2 & r=2 & r=3
\end{array}$\\
\caption{Conformally immersed Lawson surfaces. Left: Clifford torus. Middle: Klein bottle.
Right: Torus.}
\label{figureLawson}
\end{figure}

Conformal coordinates $(u,v)$ for $y$ are defined by setting
$v=\int_0^{\tilde v} (\cos ^2  w + r^2 \sin ^2 w)^{-1/2} \dd w$,
and  $u=u$.
Setting $R=\sqrt{\cos ^2\tilde v +r^2 \sin^2 \tilde v}$, a canonical lift and frame are given by:
\beqas
Y&=& (1, y),  \quad  \quad \quad \hat Y =\frac{1}{2}(1, -y),\\
P_1&=& \frac{1}{R}
    \left(0, -\sin u \cos \tilde v, \cos u \cos \tilde v, -r \sin ru \, \sin \tilde v, r \cos ru \, \sin \tilde v \right),\\
P_2 &=& (0, -\cos u \sin \tilde v, -\sin u \sin \tilde v, \cos ru \, \cos \tilde v, \sin ru \, \cos \tilde v),\\
\psi &=& \frac{1}{R}
    \left(0, -r \sin u \sin \tilde v, r \cos u \sin \tilde v, \sin ru \, \cos \tilde v, -\cos ru \, \cos \tilde v \right).\\
\eeqas
The restriction of this frame to $v=\tilde v=0$ is precisely
the frame given in  Example \ref{example1}, with $\theta^\prime=r$.
In particular the boundary potentials for Lawson's minimal surfaces $\tau_{m,\ell}$ are
given by
\[
(\mu,k,\rho) = (0,i\ell/(2m),-1/2).
\]
For the case that $r$ is not rational, one obtains an immersed cylinder.
Figure \ref{figureLawson} shows three examples computed from these potentials.
\end{example}

\section{Equivariant Surfaces} \label{equisection}
The Lawson-type surfaces of the previous example are special cases of
 Willmore surfaces invariant under the action of a $1$-parameter subgroup of $SO(4)$. More generally, by an \emph{equivariant} surface we mean one that is invariant under the action of a $1$-parameter subgroup of the M\"obius group $SO^+(1,4)$.  Such a subgroup necessarily sits inside either a copy of $SO(4)$ or of $SO^+(1,3)$, the isometry
groups of $\SSS^3$ and $\HHH^3$ respectively.
We will consider the $SO(4)$ case first, which we will call \emph{$SO(4)$-equivariant} surfaces.
Up to conjugation in $SO(4)$, such a subgroup acts on $(z,w) \in \SSS^3\subset {\mathbb C}^2$ by $(z,w) \mapsto (e^{it} z, e^{i r t} w)$, where $r \in \real$. The case $r=0$ corresponds to surfaces of revolution,
and $r= 1$ corresponds to Hopf cylinders.

\subsection{Criteria for minimality in space forms}
We are interested to distinguish those Willmore surfaces that are
``non-minimal" in the sense that they are not M\"obius equivalent
to a minimal surface in some space form. For equivariant surfaces,
the criteria is given in the lemma below.

We first remark that a standard argument \cite{BuKi} shows that a surface is equivariant, with the curve
$v=0$ an equivariant curve,  if and only if
the corresponding holomorphic potential depends only on $v$. This means that the Bj\"orling potentials
corresponding to equivariant surfaces are exactly those with $\mu$, $k$ and $\rho$ constant.  See also the direct argument below in Section \ref{sorsection}.

First we recall a well-known description of Willmore surfaces being minimal surfaces in some space forms (see for example, Page 377 of \cite{Helein}):
\begin{lemma}\label{lemma-minimal} Let  $y$ be a Willmore surface in $\SSS^n$ with $Y$ and $\hat{Y}$ a lift of itself and its dual surface. Then $y$ is M\"obius equivalent to  a minimal
surface in some $n$-dimensional space form if and only if there exist two real functions $a$ and $b$ such that $aY+b\hat{Y}\neq 0$ is constant. Moreover, the space form is
  \begin{enumerate}
	 \item $\SSS^n(r)$  if and only if $\langle aY+b\hat{Y}, aY+b\hat{Y}\rangle=-r^2$;
	  \item $\real^n$  if and only if  $\langle aY+b\hat{Y}, aY+b\hat{Y}\rangle=0$ if only if $[\hat{Y}]$ is constant;
		\item  ${\mathbb H}^n(r)$ if and only if   $\langle aY+b\hat{Y}, aY+b\hat{Y}\rangle=r^2$.
	\end{enumerate}
\end{lemma}
Applying this to equivariant Willmore surfaces in $\SSS^3$, we have
\begin{lemma} \label{minimalcharthm}
Let  $y$ be an equivariant Willmore surface generated by the boundary potential corresponding to the constants $(\mu_1,\mu_2,k_1,k_2,\rho_1,\rho_2)$.  Then
\begin{enumerate}
	 \item \label{minimalcharthm1}
	The surface $y$ is M\"obius equivalent to  a minimal
surface in $\real^3$  if and only if $\rho_1=\rho_2=0$;
	  \item \label{minimalcharthm2}
		The surface $y$ is M\"obius equivalent to  a minimal
surface in $\SSS^3$  if and only if $\mu_1=\rho_2=0 $ and $\rho_1<0$;
		\item  \label{minimalcharthm3}
		The surface $y$ is M\"obius equivalent to  a minimal
surface in $\mathbb H^3$ if and only if $\mu_1=\rho_2=0 $ and $\rho_1>0$.
	\end{enumerate}
\end{lemma}
\begin{proof}
\ref{minimalcharthm1} $~~$ By Lemma \ref{lemma-minimal} if $y$ is M\"obius equivalent to  a minimal
surface in $\real^3$, then $[\hat{Y}]$ is constant. Hence by \eqref{eq-moving-y and dual}, $\rho_1=\rho_2=0$.

Conversely, if $\rho_1=\rho_2=0$, then
\[
\mathcal{A}_1=\left(
  \begin{array}{cc}    0 & \hat{B}_1 \\
                              -\hat{B}_1^tI_{1,1} & 0 \\
                               \end{array}
             \right)\ \hbox{ with } \hat{B}_1=\left(   \begin{array}{cc}
          \hat{b}_1^t  \\
							\hat{b}_1^t  \\
									\end{array}
													\right).
\]
By simple computation one will see that $\hat{B}_1$ being of the above form
 is conjugation invariant.
So let $F=(e_0,e_1,e_2,e_3,e_4)$ be the extended frame derived from $\Xi$. Then the $B_1$ part of the Maurer-Cartan form of $F$ has the same form, which means that
$\hat{Y}_z=\frac{1}{\sqrt{2}}( e_{-1}+ e_0)_z=0 \mod \hat{Y}$. By Lemma \ref{lemma-minimal}, $y$ is M\"obius equivalent to  a minimal
surface in $\real^3$.

\ref{minimalcharthm2} $~~$  By Lemma  \ref{lemma-minimal}, if $y$ is M\"obius equivalent to  a minimal
surface in $\SSS^3$, then there exist two real functions $a$ and $b$ such that
\[
(aY+b\hat{Y})_u=0,\ \hbox{ and }\ \langle aY+b\hat{Y}, aY+b\hat{Y}\rangle=-2ab=-r^2.
\]
By \eqref{eq-moving-y and dual}, $\mu_1=\rho_2=0, \ \hbox{ and }\   a+b\rho_1=0.$
Since $ab=r^2>0$, $\rho_1<0$.

Conversely, if $\mu_1=\rho_2=0$ and $\rho_1<0$, there exists a unique $\theta_0\in\mathbb{R}$ such that $(1+\rho_1)\cosh\theta_0+(1-\rho_1)\sinh\theta_0=0.$ Let
\[
T=\hbox{diag}(T_1,I_3), \hbox{ with } T_1=\left(
                                              \begin{array}{cc}
                                                \cosh\theta_0 & \sinh\theta_0 \\
                                                \sinh\theta_0 & \cosh\theta_0 \\
                                              \end{array}
                                            \right).
\]
Then the first row and column of $\tilde\Xi=T\Xi T^{-1}$ are both
zero.
That is, $\tilde{\Xi}$ induces a  conformal  harmonic map into $SO(4)/SO(3)=\SSS^3$, which means that the surfaces induced by $\tilde\Xi$ are M\"obius equivalent to  some minimal
surfaces in $\SSS^3$. Let $F$ be the extended frame of $\Xi$. So $\tilde{F}=TFT^{-1}$ is the extended frame of $\tilde\Xi$ and hence $y$ is
 M\"obius equivalent to  some minimal
surface  in $\SSS^3$.

The proof of \ref{minimalcharthm3} is the same as \ref{minimalcharthm2}, and is left to the interested reader.
\end{proof}

\subsection{Surfaces of revolution in $\SSS^3$}  \label{sorsection}
A \emph{rotational surface} in $\SSS^3$ is  an equivariant surface
where the $1$-parameter subgroup fixes a geodesic in $\SSS^3$, or, equivalently
 fixes a plane in $\real^4$.
Without loss of
generality, we can take the geodesic to be the unit circle in the plane $E_3 \wedge E_4$,
so that the action is $R_t (z,w) = (e^{it}z,w)$.
A point on the surface that is not a fixed point of the action is (after a rotation in the fixed plane $E_3\wedge E_4$) of the form $(a \cos \theta, a \sin \theta, b,0)$, where $a^2+b^2=1$,
and $a \neq 0$. Applying $R_t$, the
surface thus contains the curve $\gamma(t)=(a \cos t, a \sin t, b,0)$, and we write our initial curve
as
\[
y(u)= a f(u) + b E_3, \quad f(u)=\cos u E_1 + \sin u E_2, \quad a^2+b^2=1,
\quad a \neq 0.
\]
The surface normal along this curve must be of the form
\[
n(u) = -bc f + ac E_3 + d E_4,
\]
 where $c^2+d^2$=1, and the assumption that the surface
is invariant under $R_u$ means that $c$ and $d$ are constant.  The starting point for the construction is the canonical lift  $Y$ of $y$ and a general $R_u$-invariant lift $\psi$ of $n$:
\beqa \label{circle_normal_sys}
Y= \frac{1}{a}E_0 + f + \frac{b}{a}E_3, \quad \quad
\psi= (0,n)+hY
\eeqa
with
\[
f(u)=\cos t E_1 + \sin t E_2, \quad  a^2+b^2=1, \quad c^2+d^2=1, \quad a \neq 0, \quad h \in \real \nonumber,
\]
where $a$, $b$, $c$, $d$ and are constant, and the constant $h$ will be the value of the mean
curvature along the curve. We expect another parameter to appear
because we have not yet chosen $\hat Y$, but we begin by finding all possible solutions to
(\ref{circle_normal_sys}), and then identify those that are equivariant.

The last equation of (\ref{eq-moving-y and dual})  becomes:
\[
(h-bc) f^\prime = \psi^\prime = -2k_1 P_1 + 2 k_2 P_2.
\]
If $h-bc=0$ then we must have $k_1=k_2=0$ along the whole curve, and hence the curve is a line of umbilics.
If the surface is not totally umbilic, we can choose a different parallel curve as our initial curve
for the Bj\"orling problem. Hence, we assume that $h\neq bc$.  In this case, $\psi^\prime \neq 0$,
and we necessarily
have $\textup{span}(P_1, P_2, \psi) = \textup{span}(\psi, \psi^\prime, V)$,
where $V$ depends on the choice of $\hat Y$. We can therefore choose $P_1 = \psi^\prime/|\psi^\prime|$,
that is:
\[
P_1 = f^\prime, \quad k_1 = \frac{\beta}{2}, \quad k_2=0, \quad \mu_1 =0,
\quad \quad \beta:= bc-h.
\]
From the second and third equation of (\ref{eq-moving-y and dual}), one obtains
\beqas
\hat Y^\prime  & = &\rho_1 P_1 + \rho_2 P_2, \\
 P_1^\prime & = & -f = \mu_2 P_2 + \beta \psi + \hat Y + \rho_1 Y.
\eeqas
Differentiating the expression $\hat Y = - \mu_2 P_2 -f -\beta \psi -\rho_1 Y$, we have
\[
\rho_1 P_1 + \rho_2 P_2 = \hat Y^\prime = -\mu_2^\prime P_2 - \mu_2 P_2^\prime - f^\prime + \beta^2 f^\prime - \rho_1^\prime Y-\rho_1 P_1.
\]
The fourth equation of (\ref{eq-moving-y and dual}) is $P_2^\prime = -\mu_2 P_1 + \rho_2 Y$.
Inserting this above, we end up with
\[
P_1(2\rho_1 - \mu_2^2+1-\beta^2) + P_2(\rho_2 + \mu_2^\prime) + Y(\rho_1^\prime + \mu_2 \rho_2) = 0.
\]
The  vanishing of the coefficients of $P_1$, $P_2$ and $Y$ above implies that
\[
\rho_1 = \frac{1}{2}\left(\mu_2^2 + \beta^2-1\right), \quad \rho_2 = -\mu_2^\prime.
\]
The third equation, $\rho_1^\prime = -\mu_2 \rho_2$, gives nothing new, and so we retain the
function $\mu_2$ as a parameter $m$.

In summary, all possible Willmore surfaces containing the curve and surface normal specified at
(\ref{circle_normal_sys}) are given by the boundary potential data
\[
(\mu, k, \rho) = \left(im(u), \,\, \frac{\beta}{2}, \,\, \frac{1}{2}(m(u)^2+\beta^2-1)-im^\prime(u)\right),
\]
where $m(u)$ is an arbitrary function of $u$.
Three examples are computed numerically and displayed in
Figure \ref{figure1}. All have the same value for $\beta$, but with
respectively $m(u)=e^{u-\pi/2}$, $m(u) = 2\cos^2(2u)$ and
$m(u)=-1$.  An interesting result of Palmer \cite{palmer} shows
that such a Willmore surface, i.e. containing a circle and intersecting the plane of the circle with constant contact angle, cannot enclose a topological disc, unless it is part of a sphere or
a plane.

Only the last or our examples is a surface of revolution, because
we have not yet taken into account that all the geometry of the surface should be
invariant under the action of $T(u)$.  In that case, the dual surface $\hat Y$, which
is unique, must also be invariant. This, combined with the invariance of
 $P_1$ and $\psi$ implies that the vector $P_2$ is invariant too. Noting that
$\langle P_2, P_1 \rangle = \langle P_2, f^\prime \rangle=0$, this means we can write
\[
P_2 = A E_0 + B f + C E_3 + D E_4,
\]
where $A$,$B$, $C$ and $E$ are all constants. Differentiating this, the fourth equation from
(\ref{eq-moving-y and dual}) is
\[
BP_1 = B f^\prime = P_2^\prime = -m P_1 + \rho_2 Y,
\]
from which we conclude that $m=-B$ is constant and $\rho_2 = 0$.
Hence, we have the characterization:
\begin{theorem}
All Willmore surfaces of revolution in $\SSS^3$
 are given by the boundary potentials with data:
\[
(\mu, k, \rho) = \left(im, \,\, \frac{\beta}{2}, \,\, \frac{1}{2}(m^2+\beta^2-1)\right),
\quad \beta \in \real, \quad m \in \real,
\]
where $\beta=bc-h$ if $b$ and $c$ are chosen as described above,
 and $h$ is the value of the mean curvature along the
initial parallel.
\end{theorem}
\begin{proof}
We have already shown this for the case $\beta \neq 0$ and non-totally umbilic surfaces.
If $\beta=0$ then the last row and column of the potential are zero, and so the surface
is an immersion into a totally geodesic sphere $\SSS^2 \subset \SSS^3$. Conversely,
The only totally umbilic surface of revolution in $\SSS^3$ is the totally geodesic $2$-sphere.
\end{proof}

\begin{figure}[ht]
\centering
$
\begin{array}{cccc}
\includegraphics[height=26mm]{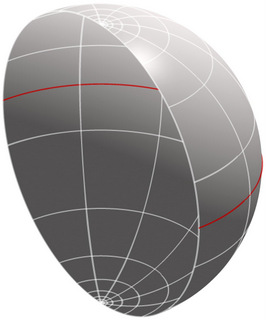} \,\,  &  \,\,
\includegraphics[height=26mm]{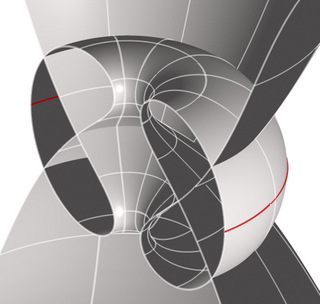}  \,\,  &  \,\,
\includegraphics[height=26mm]{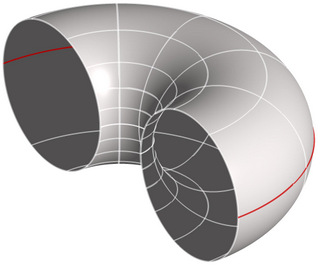}   \,\, & \,\,
\includegraphics[height=26mm]{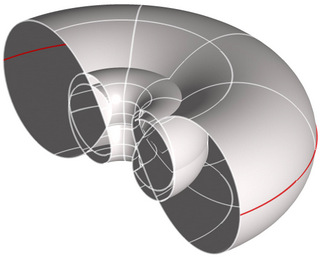}  \\
\beta=0 & \beta=1/4 & \beta=1/\sqrt{2} & \beta=7/8 \vspace{1ex}\\
\end{array}$\\
$
\begin{array}{cccc}
\includegraphics[height=26mm]{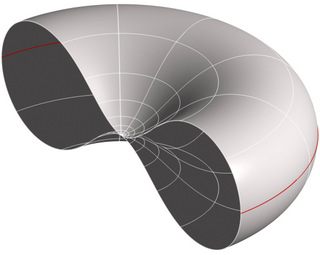}   \,\,  &  \,\,
\includegraphics[height=26mm]{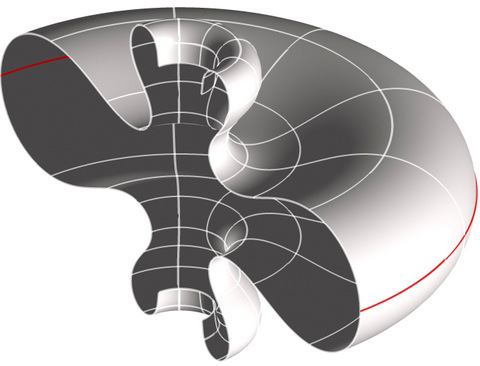}  \,\,  &  \,\,
\includegraphics[height=26mm]{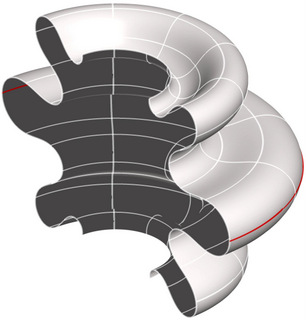}   \,\,  &  \,\,
\includegraphics[height=26mm]{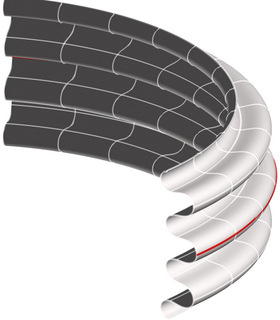}      \\
\beta=1  & \beta=1.5 & \beta=3 & \beta=15 \vspace{1ex} \\
\end{array}
$
\caption{Examples of Willmore surfaces of revolution. All are computed with $m=0$.
Surfaces are stereographically projected from the point $(0,0,0,1)$. The first four
are congruent to minimal surfaces in $\SSS^3$, the fifth  to
a catenoid, and the last three to minimal surfaces in $\HHH^3$ (Theorem \ref{equi_char_thm}).
}
\label{figureSOR1}
\end{figure}

\begin{remark}
If one is only interested in rotational surfaces up to M\"obius equivalent
then all solutions are obtained by integrating the above potential with the identity as initial condition.
To plot the surface with a suitable projection that shows the relevant symmetry,
we then premultiply the solution by the initial condition
$F_0(u_0)=\left( (Y_0+\hat Y_0)/\sqrt{2}, (-Y_0 + \hat Y_0)/\sqrt{2}, P_1, P_2,\psi_0\right)\big|_{u=u_0}$ corresponding to
a definite choice of $b$ and $c$.
For the case $\beta=0$ one only obtains totally geodesic spheres, so the initial condition is
not important. Hence all possible real values of $\beta$ are covered by taking
 $b=c=0$, $a=1$, $d=-1$ and $h$ arbitrary.
The examples shown in Figures \ref{figureSOR1} and \ref{figureSOR2}
 are computed numerically, applying this initial condition, and then stereographically projected from the point $(0,0,0,1)$.
\end{remark}

\begin{figure}[ht]
\centering
$
\begin{array}{cccc}
\includegraphics[height=26mm]{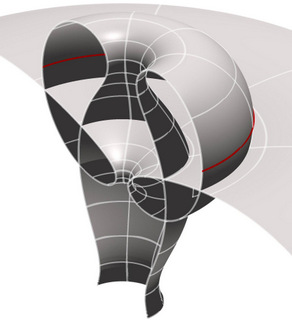} \,  &  \,
\includegraphics[height=26mm]{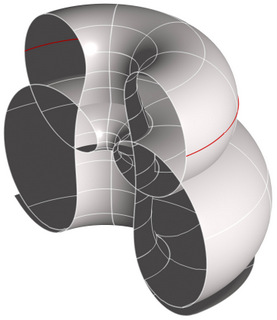}  \,  &  \,
\includegraphics[height=26mm]{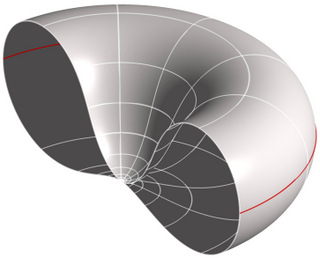}   \, & \,
\includegraphics[height=26mm]{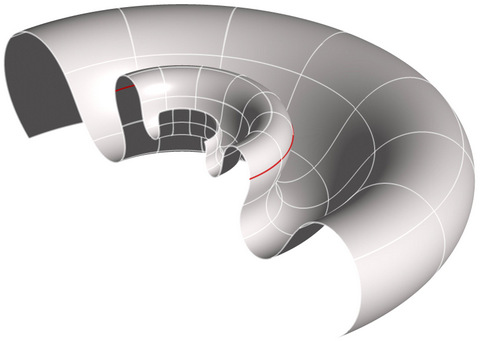}  \\
(0.5,0.25) & (0.5,\sqrt{3}/4)& (0.5,\sqrt{3}/2)& (5,1) \vspace{1ex}\\
\end{array}$\\
\caption{Surfaces of revolution with various values of $(m,\beta)$.
The non-zero value of $m$ means the surface normal along the initial curve
is not perpendicular to the axis of revolution.
}
\label{figureSOR2}
\end{figure}

\begin{remark}
On the other hand, one can obtain all solutions up to
\emph{isometric} equivalence in $\SSS^3$, if one considers all possible values of
 $b$ and $c$ in the construction and uses the correct initial condition.  To see
that this is needed for isometric equivalence, consider that if
$b=c=0$ we necessarily have $\beta=h$.  But then, for non-totally umbilic solutions
(i.e. $\beta \neq 0$)
we would need to  have $h\neq 0$.  Thus the non-trivial solutions computed with this initial condition cannot be minimal in $\SSS^3$, only M\"obius equivalent to  a
minimal surface.
\end{remark}

\subsubsection{Minimal surface representations for rotational surfaces}
It has long been known that a Willmore surface of revolution is necessarily M\"obius equivalent to  a minimal
surface in one of the three space forms (\cite{Thomsen}).  Applying Lemma \ref{minimalcharthm}, we immediately recover that result and characterize the
 corresponding potentials as follows:
\begin{theorem}  \label{equi_char_thm}
The Willmore surface of revolution corresponding to the point $(m,\beta) \in \real^2$, with $\beta \neq 0$,
 is M\"obius equivalent to  a minimal surface in:
\begin{enumerate}
\item Hyperbolic $3$-space ${\mathbb H}^3$ if and only if $m^2 + \beta^2>1$,
\item Euclidean $3$-space if and only if $m^2 + \beta^2=1$,
\item The $3$-sphere $\SSS^3$ if and only if $m^2 + \beta^2<1$.
\end{enumerate}
\end{theorem}


\subsection{Non-rotational $SO(4)$-equivariant surfaces}
We now consider $SO(4)$-equivariant surfaces that are not surfaces of revolution, namely
the isometries $(z,w)\mapsto (e^{it}z, e^{irt} w)$ where $r \neq 0$.
Let $p=(z,w) \subset {\mathbb C}^2$, with $|z|^2+|w|^2=1$ be an arbitrary point on the surface. After a rotation
of $\SSS^3$, we can assume that $z=(a,0)$ and $w=(b,0)$, where $a^2+b^2=1$.
We can therefore take the initial curve as $y=(a e^{it},be^{irt})$, with $r \neq 0$. An $SO(1,4)$ frame
for $\real_1^5 = \real \times {\mathbb C}^2$ along
the curve, invariant under the action of the subgroup, is given by
\beqas
f_0=(1,0,0), \quad f_1 =(0,e^{it}, 0), \quad f_2= (0,ie^{it},0), \\
f_3=(0,0,e^{i r t}), \quad f_4 =(0,0,i e^{irt}),
\eeqas
where, for computations, we note that $f_2=f_1^\prime$ and $f_4=f_3^\prime/r$.  Writing
all vectors as coordinate vectors in this frame,
we have the canonical lift for $y$ as
\[
Y=\frac{1}{R}(1,a,0,b,0), \quad R=\sqrt{a^2+b^2r^2}, \quad a^2+b^2=1, \quad r\neq0.
\]
The most general unit normal for the surface along $y$ give us, in the frame $f_i$,
\[
n = \left(-bc, -\frac{bdr}{R}, ac, \frac{ad}{R}\right), \quad \psi= (0,n)+hY, \quad c^2+d^2=1, \, h \in \real.
\]
where $h$, $c$ and $d$ are constant.  As with rotational surfaces,
 all of the vector fields, $\hat Y$, $P_1$ and
$P_2$ can be chosen to be invariant, and thus have constant coefficients in the basis $f_i$.
Hence all possible solutions are obtained using linear algebra.  We can write
\[
P_1 = Y^\prime + \mu_1 Y= \frac{1}{R}\left(\ell, \ell, a, \ell b, rb\right),
\]
where $\mu_1=\ell$ is constant. As in the rotational case, we assume that the surface is not totally umbilic,
implying that $\psi^\prime \neq 0$ and $\textup{span}(P_1,P_2,\psi) = \textup{span}(P_1,\psi^\prime,\psi)$.
To find $P_2$, we extend the orthonormal pair $(\psi, P_1)$ to an orthonormal basis
$(\psi, P_1, P_2)$ for  $\textup{span}(P_1, \psi, \psi^\prime)$, and find:
\[
P_2=  (0, bd, -bcR/r, -ad, acR) + \frac{abc(1-r^2)}{rR}\left( \ell, a\ell, a, b\ell, br\right)
  - \frac{h \ell}{r}(1, a, 0, b,0).
\]
It is also straightforward algebra to find the unique null vector field $\hat Y$
that is orthogonal to $P_1$, $P_2$ and $\psi$ and satisfies $\langle \hat Y, Y\rangle=-1$.
Substituting these expressions into (\ref{eq-moving-y and dual}), we finally obtain:
\begin{theorem} \label{equibjorlingthm}
All non-rotational equivariant Willmore surfaces in $\SSS^3$ are obtained from
the boundary potential $\Xi_{r,\theta, \phi,\ell,h}$, with $r\in \real\setminus \{ 0\},\
\ell,h \in \real ,\  \hbox{ and }\
\theta, \phi \in \real \mod 2\pi {\mathbb Z},$
defined as follows: write
\[
a:=\cos \theta, \quad b=\sin \theta, \quad c=\cos \phi, \quad d=\sin \phi, \quad
R=\sqrt{a^2+r^2b^2}.
\]
 The potential  $\Xi_{r,\theta, \phi,\ell,h}$ is the boundary potential with the
following data:
\beqas
\mu_1 &=&\ell, \quad \quad
\mu_2={\frac { ab\left( {r}^{2}-1 \right) \left(c\ell R+dr \right) }{rR}}+{
\frac {R\ell}{r}}h,\\
k_1&=&{\frac {abc \left( 1-{r}^{2} \right) }{2R}}-\frac{1}{2}h, \quad \quad
k_2 = {\frac {r}{2R}},\\
\rho_1 &=&
-\,\frac{R^2}{2}+{\frac {{a}^{2}{b}^{2
}cd\ell \left( {r}^{2}-1 \right) ^{2}}{rR}}+  \,{\frac {{\ell}^{2} \left( {a}^
{2}{b}^{2}{c}^{2} \left( {r}^{2}-1 \right) ^{2}-{r}^{2} \right) }{2{r}^
{2}}} \\
&&
+ \, hab
 \left( {r}^{2}-1 \right)  \left( {\frac {Rc{\ell}^{2}}{{r}^{2}}}+{\frac
{d\ell}{r}}+{\frac {c}{R}} \right) +
\frac{{h}^{2}}{2} \left({\frac {{R}^{2}{\ell}^{2}}{{r}^{2}}}+1 \right),
\\
\rho_2 &=&
 {\frac {a b d\ell
 \left( {r}^{2}-1 \right)}{R}}+{\frac {abc{\ell}^{2} \left( {r}^{2}-1
 \right) }{r}} +h \left( {\frac {{\ell}^{2}R}{r}}+{\frac {r}{R}} \right)
.
\eeqas
\end{theorem}

\begin{figure}[ht]
\centering
$
\begin{array}{ccc}
\includegraphics[height=36mm]{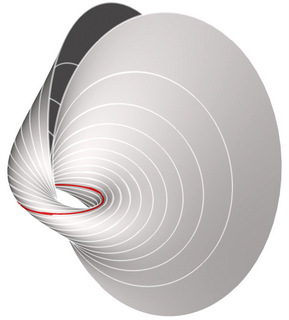} \quad &  \quad
\includegraphics[height=36mm]{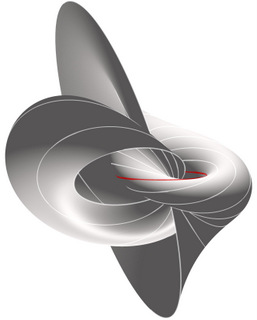} \quad &  \quad
\includegraphics[height=36mm]{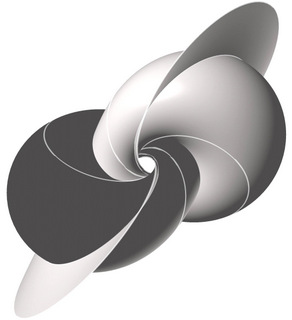}   \\
\ell=0, \,\,h=1  &   \ell=1, \,\, h=1   &   \ell=1, \,\, h=0 \vspace{1ex}\\
\end{array}$\\
\caption{Examples of Willmore Hopf cylinders. }
\label{figurehopf}
\end{figure}

\subsection{Special classes of non-rotational surfaces}
\subsubsection{Willmore Hopf cylinders, Case $r=1$:} Here the data simplifies to
 \[
(\mu_1,\mu_2,k_1,k_2,\rho_1,\rho_2) =
\left( \ell,\,h\ell, \,-\frac{h}{2}, \, \frac{1}{2}, \, \frac{{h}^{2}({\ell}^{2}+1)-\ell^2-1}{2}, \,h
 \left( {\ell}^{2}+1 \right) \right),
\]
which only depends on $h$ and $\ell$.  Hence there is a two parameter family of Willmore Hopf cylinders.  According to Lemma \ref{minimalcharthm}, the surface is M\"obius equivalent to  a minimal
surface in some space form if and only if
 $\ell=h=0$, in which case the data is of the form
$(0,0,0,1/2,-1/2,0)$,
a Clifford torus in $\SSS^3$. Otherwise, the surface is not minimal. This re-derives Proposition 2 of \cite{Pinkall1985}.

\begin{figure}[ht]
\centering
$
\begin{array}{ccc}
\includegraphics[height=30mm]{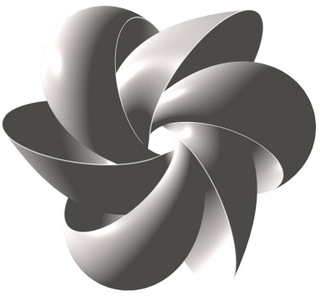} \quad & \quad
\includegraphics[height=30mm]{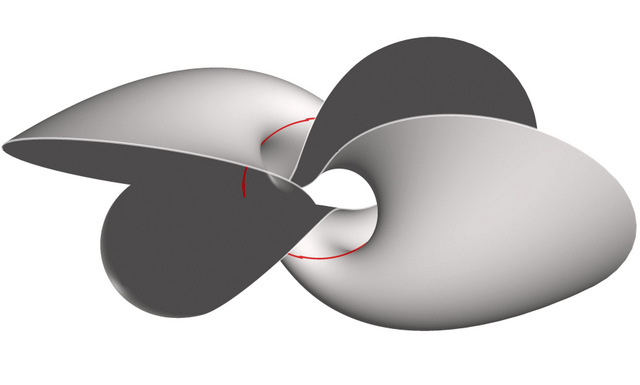} \quad & \quad
\includegraphics[height=30mm]{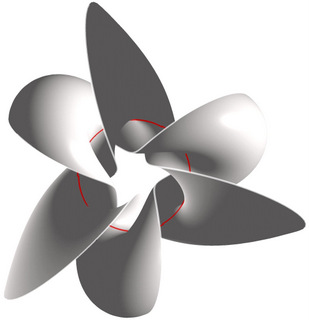}   \\
  r=\frac{3}{2}  &   r=2 & r=3 \vspace{1ex}\\
\end{array}$\\
\caption{Equivariant Willmore cylinders containing an equator (Section \ref{equicircsection}).
All have $\ell=h=1$. The value of $r$ is the number of times that the normal rotates
around the circle in one revolution. The surface is a cylinder if $r$ is rational. }
\label{figureequicirc}
\end{figure}

\begin{figure}[ht]
\centering
$
\begin{array}{cccc}
\includegraphics[height=36mm]{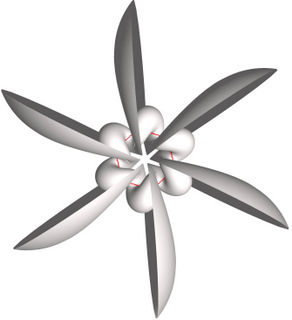} \quad &  \quad
\includegraphics[height=36mm]{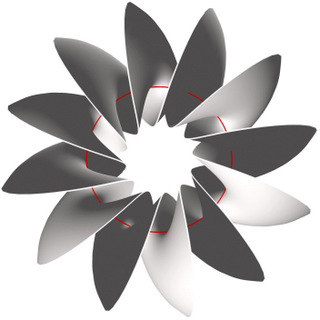} \quad &  \quad
\includegraphics[height=36mm]{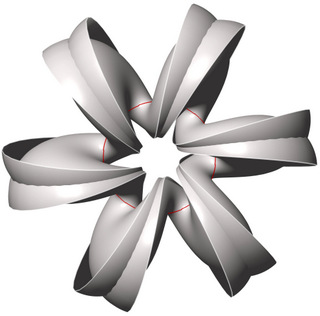}   \\
(\ell,h)=(1,1) &   (\ell,h)=(1,0)  &   (\ell,h)=(0,1) \vspace{1ex}\\
\end{array}$\\
\caption{The effect of changing the value of $h$ and $\ell$. All
have $r=6$. }
\label{figureequicirc2}
\end{figure}

\subsubsection{Equivariant surfaces containing an equator, Case $a=1$, $b=0$:}
In this case, the data $(\mu_1,\mu_2,k_1,k_2,\rho_1,\rho_2)$ are equal to
\label{equicircsection}
 \[
\left(\ell,{\frac {\ell Rh}{r}},-\frac{h}{2},{\frac {r}{2R}},\,{\frac {{h}^{2}
 \left( {\ell}^{2}+{r}^{2} \right) -{r}^{2}(1+\ell^2)}{2{R}^{2}{r}^{2}}},{\frac {
 \left( {\ell}^{2}+{r}^{2} \right) h}{Rr}}\right).
\]
The surface is minimal if and only if $h=\ell=0$, and then the data
reduces to  $(0,0,0,r/2,-1/2,0)$, the Lawson-type surfaces of Example \ref{lawsonex}.  Non-minimal examples are shown in Figures \ref{figureequicirc}
and \ref{figureequicirc2}.

\subsection{$SO(4)$-equivariant minimal surfaces}
\begin{theorem}
If a non-rotational $SO(4)$equivariant Willmore surface in $\SSS^3$  is M\"obius equivalent to  a minimal surface in some space form, that space form is necessarily $\SSS^3$. The boundary potential is given by the following
data, where $a,b,c,d,r$ and $R$ are as in Theorem \ref{equibjorlingthm}:
\[
(\mu_1,\mu_2,k_1,k_2,\rho_1,\rho_2)=
\left(0, \, \frac{abd \left( {r}^{2}-1 \right) }{R}, \, -\,\frac {abc
 \left( {r}^{2}-1 \right) }{2R}, \, \frac{r}{2R}, \, -\, \frac{{R}^{2}}{2},0 \right)
\]
\end{theorem}
\begin{proof}
Considering Lemma \ref{minimalcharthm}, note that
if the surface is minimal in $\real^3$, so that $[\hat Y]$ is
constant, we can assume, at least locally, that $\hat Y$ is
constant, so that $\mu_1$ is zero, as it is also for minimal
surfaces in the other two space forms.
  Inserting $\ell=\mu_1=0$ into the potential
given at Theorem \ref{equibjorlingthm}, we obtain the potential data:
\[
\left(0,{\frac {ab d\left( {r}^{2}-1 \right)}{R}},
\,{\frac {abc(1-r^2)}{2R}}-\frac{H}{2},  \,{\frac {r}{2R}},
\,\frac{Habc({r}^{2}-1)}{R}+\frac{H^2-R^2}{2},
\, {\frac {rH}{R}}\right)
\]
in particular $\rho_2= rH/R$, and this is zero
if and only if $H=0$. With $\ell=H=0$, the data reduce to that
given in the statement of the theorem.  Since $r\neq 0$, we have $\rho_1<0$ and so the surface can only be minimal in
$\SSS^3$.
\end{proof}



\section{$SO(1,3)$-Equivariant Surfaces} \label{hypequisection}
Given a lift $Y$ of a Willmore surface $y$ in $\SSS^3$ to the light cone in $\real_1^5$, any of the projections to $\HHH^3 \subset \real_1^4$, for example
\[
(Y_0,Y_1,...,Y_4) \mapsto (Y_0,Y_1,Y_2,Y_3)/Y_4,
\]
 gives a Willmore surface (possibly with singularities) in $\HHH^3$, M\"obius equivalent to  $y$. Each choice of subgroup $SO(1,3)$ in $SO(1,4)$ corresponds to one of these
projections.  For definiteness, we choose the projection above, which corresponds
to the subgroup $SO(1,3) \times \{1\}$.  Since we have already considered the
subgroup $\SSS^1$, the only $1$-parameter subgroups left are of the form
\[
\exp \left\{ \textup{diag}\left(\bbar{cc} 0 & t \\ t & 0 \ebar ,
          \bbar{cc} 0 & rt \\ -rt & 0 \ebar, 0 \right) \right\}, \quad r \in \real.
\]
After an action of
$SO(1,1) \times SO(2)$, and a rescaling so that $\langle Y^\prime, Y^\prime \rangle=1$, we can  assume the initial curve is of the form
\[
T(u)Y(0) = \bbar {ccccc}
        \cosh u & \sinh u & 0 & 0 & 0\\
				\sinh u & \cosh u & 0 & 0 & 0\\
				0 & 0  & \cos r u & -\sin ru  & 0 \\
				0 & 0 &  \sin ru &  \cos ru & 0 \\
				0 & 0 & 0 & 0 & 1 \ebar
				\bbar{c} a\\ 0  \\ b \\ 0 \\ c \ebar,
				~~
				\begin{array} {c} a^2+r^2b^2=1, \\
				 c^2=a^2-b^2.
				\end{array}
\]
The general solution can be found as in the $SO(4)$ case. To simplify matters, we
discuss two interesting cases: one case which includes  the hyperbolic
analogue of rotational surfaces in the next
subsection, and then the case $r=1$ in the  following subsection.

\subsection{Case $a=1, \,\, b=0,\,\,c=1$:} This case includes, but is
not restricted to, the case $r=0$, because if $r$ is zero then $a=\pm 1$, and the lower right part of $T(u)$
is the $3 \times 3$ identity matrix $I_3$. In this case, there are many possible hyperbolic spaces on which
$T(u)$ acts isometrically, and we can freely rotate among the last three
coordinates without losing any generality. Hence we can assume that our
initial point is $(1, 0, 0, 0, 1)$, that is, $a=c=1$ and $b=0$.
 A suitable invariant frame along the curve is given by $\xi_i:=T(u)(E_i)$, namely:
\beqas
\xi_0=(\cosh u, \sinh u, 0,0,0), \quad \xi_1=(\sinh u, \cosh u, 0,0,0), \\
\xi_2 = (0,0,\cos ru, \sin ru, 0), \quad
\xi_3 = (0,0,-\sin ru, \cos ru, 0), \quad \xi_4 = (0,0,0,0,1).
\eeqas
Writing vectors as coordinate vectors in the frame $\xi_i$, we find, for $a=1$, $b=0$, the frame:
\beq  \label{bzeroframe}
\begin{split}
Y= (1,0,0,0,1), \quad \hat Y = \left(\frac{h^2+1}{2}, 0, h \cos \theta, h \sin \theta, \frac{h^2-1}{2}\right), \\
P_1 = (0,1,0,0,0),\quad
P_2 = (0, 0, \sin \theta, -\cos \theta, 0),\\
\psi = (h, 0, \cos \theta, \sin \theta, h )
\end{split}
\eeq
where $h$ and $\theta$ are arbitrary real constants.  Using the equations (\ref{eq-moving-y and dual}), we find the potential data:
\[
(\mu_1,\,\,\mu_2,\,\,k_1,\,\,k_2,\,\, \rho_1,\,\, \rho_2) =
   \left( 0, \,\, 0,\,\, - \frac{h}{2}, \,\, -\frac{r}{2},\,\, \frac{h^2+1}{2}, \,\,-hr \right).
\]
Note that these surfaces are congruent to minimal surfaces in $\HHH^3$ if and only if $hr=0$.
If $hr \neq0$ then they are not congruent to a minimal surface in any space form.

\subsubsection{The minimal case, $hr=0$}
Note that a discussion of rotational minimal surfaces in $\HHH^3$ can be found in \cite{doDa}.
If both $h$ and $r$ are zero then the surface is a totally umbilic sphere. Other than this
there are two types: surfaces with $r=0$, which are a hyperbolic version of surfaces of
revolution, and surfaces with $h=0$, a hyperbolic analogue of the Lawson-type surfaces
in Example \ref{lawsonex}.

\begin{figure}[ht]
\centering
$
  \begin{array}{cccc}
  \includegraphics[height=21mm]{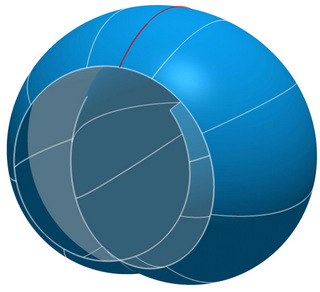} \,  &  \,
  \includegraphics[height=21mm]{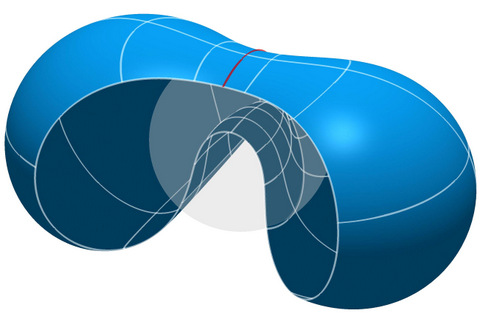}  \,  &  \,
  \includegraphics[height=21mm]{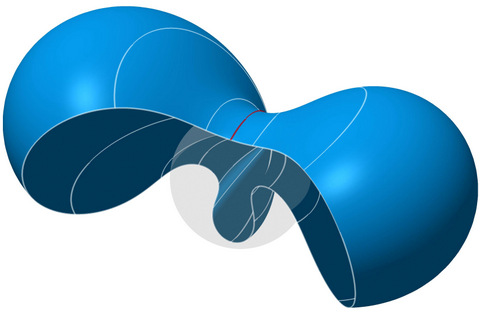}   \,  &  \,
   \includegraphics[height=21mm]{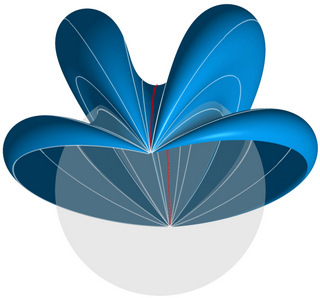}  \vspace{2ex}\\
	  \includegraphics[height=21mm]{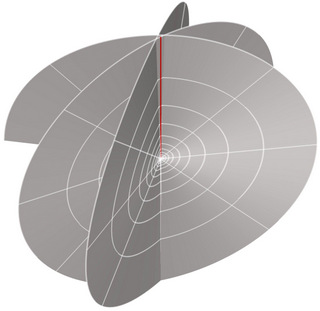}  \,  &  \,
  \includegraphics[height=21mm]{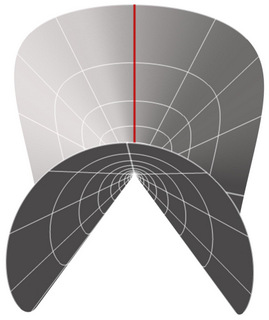}  \,  &  \,
  \includegraphics[height=21mm]{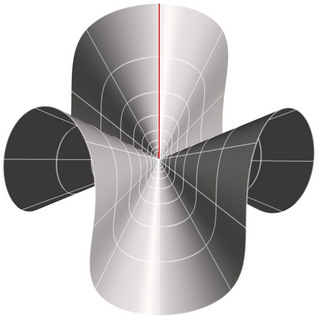}  \,  &  \,
   \includegraphics[height=21mm]{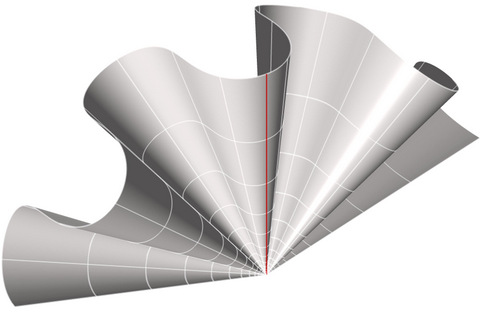}  \\
	   h=0.1 & h=0.6221 & h=1.065 & h=3
  \end{array}
$
\caption{Top: hyperbolic Willmore surfaces of revolution in $\HHH^3$ (case $r=0$), projected to the Poincar\'e ball. Bottom: M\"obius equivalent surfaces in $\SSS^3$
projected from the point $(-1,0,0,0)$. The latter are cones.}
\label{figurehypSOR}
\end{figure}

Note that in the case $r=0$, the action is by $SO(1,1) \times \{I_3\}$, so
the surfaces $(Y_0, Y_1, Y_2, Y_4)/Y_3$, and
$(Y_0, Y_1,Y_3,Y_4)/Y_2$ will also be rotationally invariant in $\HHH^3$.
Some examples from the case $r=0$ are displayed in Figure \ref{figurehypSOR}.  The initial curve is $Y(u,0) = (\cosh (u), \sinh(u), 0, 0, 1)$.
 We have plotted the projection
$(Y_0,Y_1,Y_2,Y_3,Y_4) \mapsto (Y_1,Y_2,Y_4)/(Y_0-Y_3)$, which can be regarded either as
the stereographic projection from $(0,0,1,0)$ of the surface $(Y_1,Y_2,Y_3,Y_4)/Y_0$ in
$\SSS^3$, or a Poincar\'e ball image of the surface
$(Y_0,Y_1,Y_2,Y_4)/Y_3$ in $\HHH^3$.  As surfaces in $\HHH^3$ they have several
pieces, as they pass through the boundary of the Poincar\'e ball.
A different projection of the same surfaces in
$\SSS^3$ is also shown. This corresponds to a different Willmore surface in
$\HHH^3$, which is not isometrically  equivalent, only
M\"obius equivalent.
  For certain values of $h$
(the middle two surfaces), the numerics indicate that the surface closes up in the $v$ direction.

The other type of minimal surfaces in $\HHH^3$ here are those with $h=0$, $r\neq0$. The
potential data is:
\[
(\mu, k, \rho) = \left(0, -i \frac{r}{2}, \frac{1}{2}\right),
\]
differing from the Lawson-type potentials of Example \ref{lawsonex}
only in the sign of $\rho$.  Again we have an explicit form for
the solutions: consider the surface in $\HHH^3 \subset \real_1^3 \subset \real_1^4$ given by:
\beq \label{hypLawson}
f(u,\tilde v) = \left(\cosh \tilde v \cosh u,  \,
\cosh \tilde v \sinh u,  \,  \sinh \tilde v \cos ru, \, \sinh \tilde v \sin ru \right).
\eeq
This is an analogue in $\HHH^3$ of the Lawson type surfaces, and a geodesically ruled minimal surface, that appears in \cite{doDa}.
Consider now the lift to the light cone and associated frame given by,
for $R=\sqrt{\cosh^2 \tilde v + r^2 \sinh^2 \tilde v}$:
\beqas
Y(u, \tilde v) &=& \left( f(u,\tilde v), \,\, 1 \right), \quad \quad
\hat Y (u, \tilde v) = \frac{1}{2}\left( f(u,\tilde v), \,\,   -1 \right),\\
P_1(u, \tilde v)  &=& \frac{1}{R} \left( \cosh \tilde v \sinh u,
   \cosh \tilde v \cosh u, -r \sinh \tilde v \sin ru, r \sinh \tilde v \cos ru, 0 \right), \\
P_2(u, \tilde v)  &=&  \left( \sinh \tilde v \cosh u,
  \sinh \tilde v \sinh u, \cosh \tilde v \cos ru, \cosh \tilde v \sin ru,
	  0 \right),\\
\psi (u, \tilde v)  & =& -\frac{1}{R}
\left( r \sinh \tilde v \sinh u,  r \sinh \tilde v \cosh u, \cosh \tilde v \sin ru, - \cosh \tilde v \cos ru, 0 \right).
\eeqas

\begin{figure}[ht]
\centering
$
\begin{array}{cc}
  \includegraphics[height=32mm]{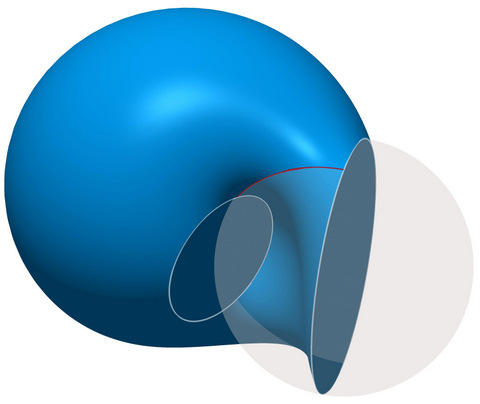} \quad &  \quad
  \includegraphics[height=32mm]{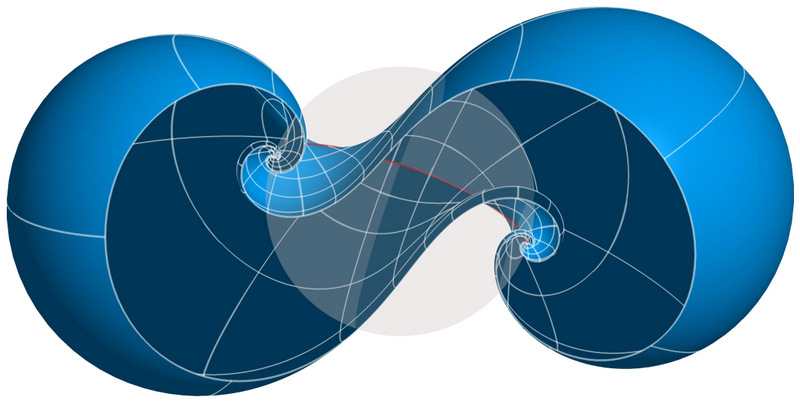} \\
  \vspace{2ex}\\
  \end{array}
$
$
\begin{array}{cccc}
  \includegraphics[height=24mm]{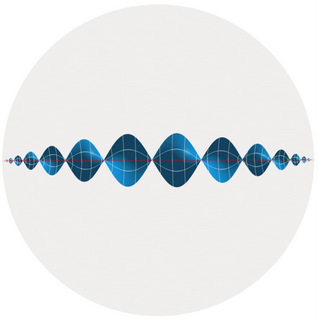} \, &  \,
  \includegraphics[height=24mm]{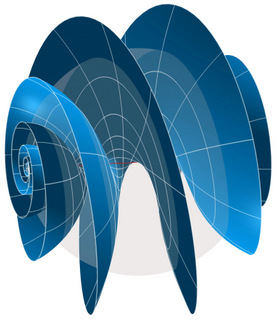}  \,  &  \,
	\includegraphics[height=24mm]{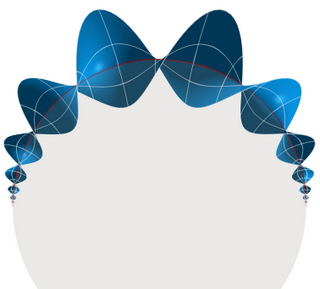}  \,  &  \,
	  \includegraphics[height=24mm]{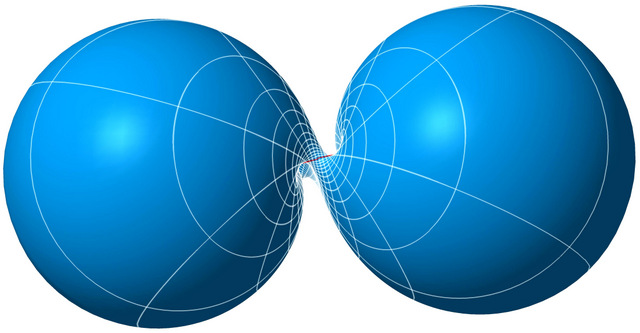}  \\
	  (0,0,0,1) & (0,0,0,1)   &  (0,0,1,0) & (0,0,1,0)
  \end{array}
$
\caption{Top: the hyperbolic Lawson surface \eqref{hypLawson2} with $r=2$,
projected from the point $(0,0,1,0)$. The $v$-curves are circles.
Bottom: two different projections of the case $r=5$. The projection from $(0,0,0,1)$ is the Poincare ball image of a minimal surface in $H^3$. The projection from $(0,0,1,0)$ is a topological cylinder.}
\label{hyplawsonr5}
\end{figure}

With respect to the coordinates $(u,v)$, where $v$ is given by
\[
v(\tilde v)= \int_0^{\tilde v}
\left( \cosh^2  \nu
 +{r}^{2}  \sinh^2  \nu \right)^{-1/2}\dd \nu,
\]
the maps $Y$ and $\hat Y$ are conformally immersed,  and canonical lifts of $f$,
by which we mean that $\langle Y, \hat Y \rangle = -1$, and $|\dd Y|^2 = |\dd z|^2$.
Additionally, $\psi_z$ is orthogonal to both $Y$ and $\hat Y$, and the frame is
orthonormal. Finally, along the curve $v=0$, this frame is nothing other than
the frame given above at (\ref{bzeroframe}), for the case $h=0$, with
$\theta=\pi/2$. The value of $\theta$ is not relevant, since it does not appear
in the potential.  Hence the
maps $Y(u,v)$, for $r\neq0$, give all the solutions for this case.

Note that the  $\tilde v$ coordinate in (\ref{hypLawson}) only gives a part of the surface,
namely that part that lies in one copy of $\HHH^3$. The map $\tilde v \mapsto v$ takes the whole
real line to a bounded open interval in $\real$.  Computing the rest of the surface
with the coordinate $v$, we find that the surface continues smoothly through the boundary.
In fact the curves $u=\textup{constant}$ are closed curves, and the surface
\beq \label{hypLawson2}
y(u, v) = \frac{1}{\cosh  \tilde v \cosh u}
 \left( \cosh  \tilde v \sinh  u,  \,  \sinh  \tilde v \cos ru,
\, \sinh \tilde v \sin ru , \, 1 \right),
\eeq
 in $\SSS^3$ is apparently a topological cylinder.

\subsubsection{The non-minimal case, $hr\neq 0$}  \label{sectionhrn0}
Examples that are not congruent to minimal surfaces are shown in Figure \ref{hypequibzero2},
where we used the projection $(Y_0,Y_1,Y_2,Y_3,Y_4) \mapsto (Y_1,Y_2,Y_3)/(Y_0+Y_4)$.
The initial curve in this projection is the straight line segment $\{(x,0,0) ~|~ -1<x<1\}$.
A different projection, $(Y_1,Y_2,Y_4)/(Y_0-Y_3)$ of the case $r=2$ is also shown
in Figure \ref{figure2}.

\begin{figure}[ht]
\centering
$
\begin{array}{ccc}
  \includegraphics[height=30mm]{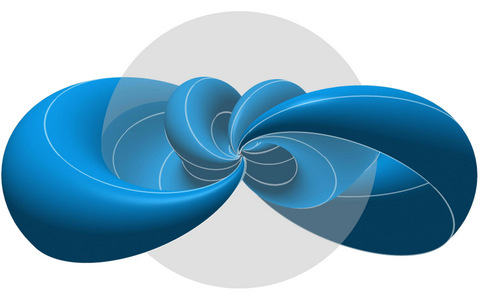} \quad &  \quad
  \includegraphics[height=30mm]{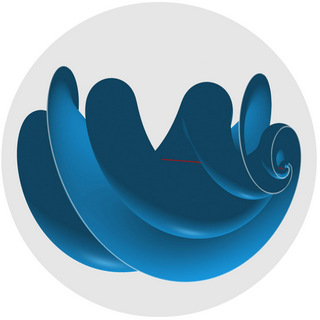}  \quad &  \quad
	  \includegraphics[height=30mm]{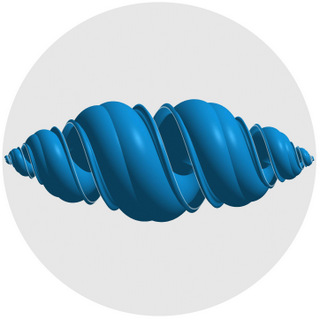}  \\
	  r=0.5 & r=2   &  r=5
  \end{array}
$
\caption{Nonminimal $SO(1,3)$-equivariant  surfaces in $\HHH^3$. All have
$b=0$ and $h=2$.  }
\label{hypequibzero2}
\end{figure}

\subsection{Hyperbolic Hopf surface analogues: Case $r=1$:}
We again write vectors as coordinate vectors in the frame $\xi_i$. The initial
curve is thus:
\[
Y=(a,0,b,0,c), \quad a^2+b^2=1, \quad c^2=a^2-b^2.
\]
After suitable isometries of the ambient space, we can assume that $a$, $b$ and $c$ are
all non-negative, so that there is a unique constant $\theta$ satisfying:
\[
a=\cos \theta , \quad b=\sin \theta , \quad c=\sqrt{\cos 2 \theta},
 \quad \quad 0\leq \theta \leq \frac{\pi}{4}.
\]

The most general choice for $\psi$ and $P_1=\psi^\prime + \mu_1 \psi$, invariant along the curve are, in the basis $\xi_i$:
\beqas
\psi &=& \left( 0, \,  -b^2 q,  \,  -\frac{c}{a}\sqrt{1-b^2 q^2}, \, abq, \, \frac{b}{a}\sqrt{1-b^2 q^2} \right) +h Y,  \\
  P_1&=& \left(0, \,a, \, 0,\, b, \,0\right) + mY,
	\quad \quad h,  \, m\, \in \real, \quad  |q| \leq \frac{1}{|b|},
\eeqas
where $m$, $q$ and $h$ are all constant.
We extend these	using linear algebra to find the most general form for
\[
P_2	= \left(-cq, \, -b \sqrt{1-b^2 q^2}, \,  0, \, a\sqrt{1-b^2q^2},\, -aq \right)
+pY,
\quad p\in \real,
\]
and finally find the unique null vector $\hat Y$ orthogonal to $P_1$, $P_2$, $\psi$ and $\psi^\prime$ satisfying $\langle \hat Y, Y \rangle = -1$.
The condition $\langle \hat Y, \psi^\prime\rangle=0$ gives a further constraint on the parameters:
\beq \label{dualityconstraint}
amh+{a}^{2}q-bcm\sqrt {1-{b}^{2}{q}^{2}}-acp=0.
\eeq
Substituting $Y$, $\hat Y$, $\psi$, $P_1$ and $P_2$ into (\ref{eq-moving-y and dual}), we obtain
the boundary potential data:
\beq  \label{hypequir1potential}
\begin{split}
(\mu_1,  \, \mu_2, \,  k_1, \, k_2) &=
  \left(m, \,\,   acq-p,  \,\,{\frac {bc\sqrt {1-b^2 {q}^{2}}}{2a}}-\frac{h}{2}, \,\,  -\frac{c}{2} \right), \\
\rho_1 &=
 {\frac {{a}^{2}{h}^{2}-2\,abc h\sqrt {1-{q}^{2}{b}^{2}}+{a}^{2}{p
}^{2}+ {c}^{4} {q}^{2}-2\,{a}^{3}cpq-{a}^{2}{m}^{2}+{c}^{2}}{2{a}^{2}}},\\
\rho_2 &=
  {\frac {{a}^{2}cmq-ach-amp- b \sqrt {1-{q}^{2}{b}^{2}}}{a}}.
\end{split}
\eeq

\begin{figure}[ht]
\centering
$
\begin{array}{cccc}
\includegraphics[height=28mm]{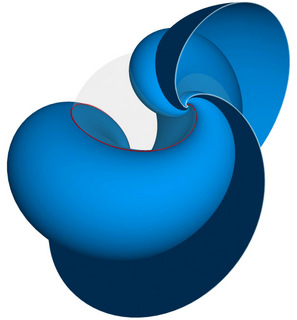} \,\, & \,\,
\includegraphics[height=28mm]{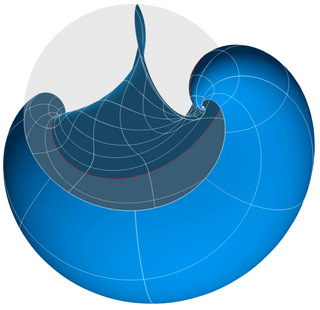} \,\, & \,\,
\includegraphics[height=28mm]{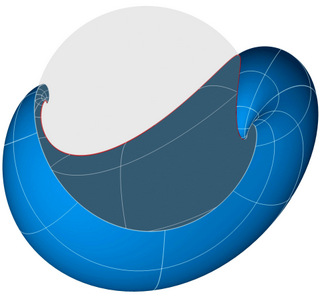} \,\, & \,\,
\includegraphics[height=28mm]{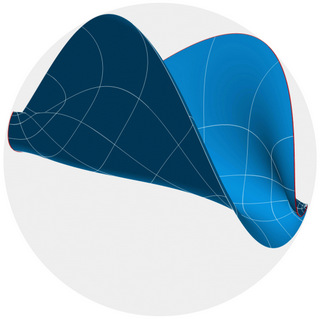}   \\
\end{array}$\\
\caption{Several partial plots of a hyperbolic Hopf-type surface. Here $c=0$
and  $h=m=p=1$. The third images shows one of the pieces outside the Poincar\'e
sphere, the fourth image one of the inside pieces. When $c=0$, the initial
curve lies on the sphere itself.  }
\label{figurehypHopf1}
\end{figure}

If $c$ is non-zero, we can eliminate $p$ by solving the constraint (\ref{dualityconstraint}) for $p$,  while if $c$ is zero, all the data simplifies
and we can eliminate $q$ instead.  We summarize this as
\begin{theorem}
All $SO(1,3)$-equivariant surfaces with $r=1$ are determined by
the boundary potentials with data as follows:
\begin{enumerate}
\item
 If $c \neq 0$, then the boundary
potential data is (\ref{hypequir1potential}), where $a= \cos \theta$,
$b=\sin \theta$, $c = \sqrt{\cos 2 \theta}$, and
\[
p={\frac {amh+{a}^{2}q -bcm \sqrt {1-{b}^{2}{q}^{2}}}{ac}}.
\]
The real parameters
$\theta$, $m$,  $q$ and $h$ are arbitrary, subject to the
conditions:
\[
0 \leq \theta < \frac{\pi}{4}, \quad
|q| \leq \frac{1}{|\sin \theta|}.
\]
\item
 If $c=0$, the boundary potentials are given by:
\[
(\mu_1, \mu_2, k_1, k_2, \rho_1, \rho_2) =
\left( m, \,\,\, -p, \, \, \,-\frac{h}{2}, \, \,0, \, \, \,
  \frac{h^2+p^2-m^2}{2}, \,\,\, -pm-\sqrt{1-h^2m^2}\right),
\]
for $h$, $m$ and $p$ arbitrary real numbers subject to the condition
$|hm|\leq 1$.
\end{enumerate}
\end{theorem}

\section{Isotropic and half-isotropic harmonic maps associated to Willmore surfaces in $\SSS^{n+2}$}  \label{section7}

 H\'elein's treatment \cite{Helein}  of  Willmore surfaces   has been generalized in \cite{Xia-Shen} to $\SSS^{n+2}$. However, the geometry inside was unclear prior to the introduction of adjoint transforms by Xiang Ma \cite{Ma2006}. One aim in this section is to clarify this interesting relationship between Willmore surfaces and isotropic harmonic maps using the language of \cite{BPP} and \cite{Ma2006}.
In Section \ref{section8}, we will use half-isotropic maps to solve the Bj\"orling problem for all Willmore surfaces in any codimension,
with or without umbilics.

\subsection{Adjoint transforms and harmonic maps associated to Willmore surfaces}

Let $Y$ be a Willmore surface as  in Section \ref{section-anotherharmonicmap}. As before, consider another  lightlike vector $\hat{Y}$ in the mean curvature sphere of $Y$, given by
\[
\hat{Y}=N+ \bar\mu Y_{z}+ \mu Y_{\bar{z}}+ \frac{1}{2}|\mu|^2Y,
\]
satisfying $\langle\hat Y, Y\rangle=-1$.
\begin{definition} \cite{Ma2006} The map into $\SSS^n$ determined by $\hat{Y}$,
defined as above, is called an adjoint transform of the Willmore surface $Y$ if the following two equations hold for $\mu$:
\begin{equation}\label{eq-h-1II}
\mu_z-\frac{\mu^2}{2}-s=0,
\mu_z-\frac{\mu^2}{2}-s=0,
\end{equation}
\begin{equation}\label{eq-h-m}
\langle D_{\bar z}\kappa+\frac{\bar\mu}{2}\kappa, D_{\bar z}\kappa+\frac{\bar\mu}{2}\kappa \rangle=\sum_j\gamma_j^2=0.
\end{equation}
\end{definition}

\begin{theorem}\cite{Ma2006}\ {\em Willmore property and existence of adjoint transform:}
  The adjoint transform $\hat{Y}$ of a Willmore surface $y$ is also a Willmore surface (may degenerate). Moreover,
\begin{enumerate}
\item
When $\langle \kappa,\kappa\rangle \equiv0$, any solution to the equation \eqref{eq-h-1II}  is a solution to both  \eqref{eq-h-1II} and \eqref{eq-h-m}.
\item When $\langle \kappa,\kappa\rangle \neq 0$ and
$\Omega \dd z^6:= \left( \langle D_{\bar{z}} \kappa,\kappa\rangle^2 -\langle \kappa,\kappa\rangle \langle D_{\bar{z}}\kappa,D_{\bar{z}}\kappa\rangle \right)  \dd z^6\neq 0,$
 there are exactly two different solutions to equation \eqref{eq-h-m}, which also solve \eqref{eq-h-1II}, that is, exactly two adjoint surfaces of $[Y]$.
\item When $\langle \kappa,\kappa\rangle \neq 0$ and $\langle D_{\bar{z}} \kappa,\kappa\rangle^2 -\langle \kappa,\kappa\rangle \langle D_{\bar{z}}\kappa,D_{\bar{z}}\kappa\rangle \equiv 0,$ there exists a unique solution to  \eqref{eq-h-m}, which also solves  \eqref{eq-h-1II}, that is, a unique adjoint surface  of $[Y]$.
\end{enumerate}
\end{theorem}

\begin{theorem}
Let  $[Y]$ be a Willmore surface. Let $\mu$ be a solution to the Riccati equation \eqref{eq-h-1II} on $U$, defining $\hat Y$
as above.
Let
$f_h:  U \to  SO^{+}(1,n+3)/(SO^+(1,1)\times SO(n+2))$ be
the map taking $p$ to  $Y(p)\wedge \hat{Y}(p)$.  Then:
\begin{enumerate}
\item (\cite{Helein}, \cite{Xia-Shen}) The map $f_h$ is harmonic,
and is called a {\em half$-$isotropic harmonic map with respect to $Y$}.
\item (\cite{Ma2006}) If $\mu$ also solves \eqref{eq-h-m},  $f_h$ is  \emph{conformally} harmonic, and is called an \emph{isotropic harmonic
map with respect to $Y$}.
\end{enumerate}
\end{theorem}

\begin{proposition}
Let  $f_h=Y\wedge\hat{Y}$ be a half$-$isotropic harmonic map.
Choose $e_1$, $e_2$ with
 $Y_{z}+\frac{\mu}{2} Y=\frac{1}{2}(e_1-ie_2)$, and
a frame $\{\psi_j,j=1,\cdots,n\}$ of the normal bundle $V^{\perp}$,
so that $\kappa=\sum_{j=1}^{n}k_j\psi_j,\ \zeta=\sum_{j=1}^{n}\gamma_j\psi_j,\ \ D_z\psi_j=\sum_{l=1}^{n}b_{jl}\psi_l,\ \ b_{jl}+b_{lj}=0.
$
Set
\[
F=\left(\frac{1}{\sqrt{2}}(Y+\hat{Y}),\frac{1}{\sqrt{2}}(-Y+\hat{Y}),e_1,e_2,\psi_1,\cdots,\psi_n\right).
\]
Then the Maurer-Cartan form
 $\alpha=F^{-1} \dd F=\alpha^\prime+\alpha^{\prime \prime}$ of $F$ has the structure:
\begin{equation}\label{eq-hm-m-c-h}
\alpha^\prime =\left(      \begin{array}{cc}
                     A_1 & B_1 \\
                     -B_1^tI_{1,1} & A_2 \\
                   \end{array}
                 \right)\dd z,
\end{equation}
with
\[
A_1=\left(
                      \begin{array}{cc}
                        0 & \frac{\mu}{2} \\
                        \frac{\mu}{2} & 0 \\
                      \end{array}
                    \right),\
 B_1=\left(
      \begin{array}{ccccccc}
                       \frac{1+\rho}{2\sqrt{2}} &  \frac{-i-i\rho}{2\sqrt{2}}  & \sqrt{2}\gamma_1 & \cdots & \sqrt{2}\gamma_n\\
                        \frac{1-\rho}{2\sqrt{2}} &  \frac{-i+i\rho}{2\sqrt{2}} &  -\sqrt{2}\gamma_1 & \cdots & -\sqrt{2}\gamma_n\\
      \end{array}
    \right)=\left(
              \begin{array}{c}
                b_1^t \\
                b_2^t \\
              \end{array}
            \right),
\]
and
\begin{equation}\label{eq-B1-h}
B_1B_1^t=2\left(\sum_{j=1}^{n}\gamma_j^2\right)\cdot {\bf E},\ \hbox{ with } {\bf E}:=\left(
                           \begin{array}{cc}
                             1 & -1 \\
                             -1 & 1 \\
                           \end{array}
                         \right).
\end{equation}
Moreover, $f_h$ is an isotropic harmonic map, and hence $\hat Y$ an adjoint
transform of $Y$,
if and only if $f_h$ is a conformally harmonic map,
  if and only if
 \begin{equation}\label{eq-B1-Ma-h}
 B_1B_1^t=0.
\end{equation}
\end{proposition}

\begin{lemma} The maps $[Y]$ and $[\hat{Y}]$ associated to a
half$-$isotropic harmonic map are a pair of dual (S-)Willmore surfaces if and only if $rank(B_1)=1$.
\end{lemma}
For any $\Psi_1 \in SO(1,1)$ there exists some $a\in\mathbb{R}^+$ such that
\begin{equation}\label{eq-conjugation invariant}
\Psi_1{\bf E}\Psi_1^t=a^2\cdot\left(
                           \begin{array}{cc}
                             1 & -1 \\
                             -1 & 1 \\
                           \end{array}
                         \right)=a^2{\bf E}.
\end{equation}
It follows that the condition \eqref{eq-B1-h} on $B_1$ is independent of the choice of frame $F$ for $f_h$.
The following theorem shows that Equation \eqref{eq-B1-h} is a good condition to characterize half$-$isotropic harmonic maps. We refer to \cite{Helein}, \cite{Xia-Shen}, \cite{Wang} for a proof.
 \begin{theorem}\label{thm-mc form2}
Let  $f:M\rightarrow SO^+(1,n+3)/(SO^+(1,1)\times SO(n))$ be a harmonic map satisfying \eqref{eq-B1-h}. Then either $f=Y\wedge \hat{Y}$ is a half$-$isotropic harmonic map associated with the Willmore surface $Y$, or
                $B_1$ is of the form
						\[ \left(
                   \begin{array}{cc}
                     b_1  &
                     -b_1 \\
                   \end{array}
                 \right)^t
								\]
                 for some $b_1$. In the latter case $[Y]$ is a constant point in $\SSS^{n+2}$.
\end{theorem}

\subsection{Harmonic maps into $SO^+(1,n+3)/(SO^+(1,1)\times SO(n+2))$}
Let $f: M\rightarrow SO^+(1,n+3)/(SO^+(1,1)\times SO(n+2))$
 be an harmonic map with a (local) lift frame $F:M\rightarrow SO^+(1,n+3)$ and the Maurer-Cartan form $\alpha=F^{-1}\dd F$. Let $z$ be a local complex coordinate of $M$. Then
\[
\alpha_0^\prime=\left(
                   \begin{array}{cc}
                     A_1 &0 \\
                     0 & A_2 \\
                   \end{array}
                 \right) \dd z,\ \ \alpha_1^\prime=\left(
                   \begin{array}{cc}
                     0 & B_1 \\
                     -B_1^tI_{1,1} & 0 \\
                   \end{array}\right) \dd z.
\]

To have a detailed discussion of half$-$isotropic and isotropic harmonic maps, we first take a look at their normalized potentials.
  \begin{theorem}(\cite{Helein}, \cite{Helein2}, \cite{Xia-Shen}) The normalized potential of a half$-$isotropic harmonic map $f=Y\wedge\hat{Y}$ is of the form
 \[
\eta=\lambda^{-1}\left(
                   \begin{array}{cc}
                    0 & \hat{B}_1 \\
                     -\hat{B}_1^tI_{1,1} & 0\\
                   \end{array}
                 \right)\dd z,
\]
with
\begin{equation}\label{eq-b1 of H-A}
\hat{B}_1\hat{B}_1^t=\hat{\gamma}{\bf E}.
\end{equation}  And if $f$ is an isotropic harmonic map, then
 \begin{equation}\label{eq-b1 of a}\hat{B}_1\hat{B}_1^t=0.
\end{equation}
Moreover, $[Y]$ and $[\hat{Y}]$ forms a pair of dual (S-)Willmore surfaces if and only if $rank(
\hat{B}_1)=1$.
\end{theorem}

\begin{proof} Let $\tilde{A}_1,$ $\tilde{A}_2$ and $\tilde{B}_1$ be the holomorphic part of $A_1,$ $A_2$ and $B_1$  respectively,
with respect to some base point $z_0$ such that $F(z_0,\lambda)=I$. So $\tilde{B}_1$ has the same form as $B_1$ and hence
$
\tilde{B}_1\tilde{B}_1^t=\tilde{\gamma}{\bf E}
$
                         for some $\hat{\gamma}$.
Let
$\Psi=\hbox{diag}\{
                            \Psi_1 ,  \Psi_2  \}$
be a solution to
\[
\Psi^{-1}d\Psi=\left(
                           \begin{array}{cc}
                            \tilde{A}_1 & 0 \\
                             0 & \tilde{A}_2 \\
                           \end{array}
                         \right)\dd z,\ ~~~\Psi(z_0)=I.
\]
By Wu's formula in Theorem \ref{thm-wu},
\[
\eta=\lambda^{-1}\Psi\left(
                           \begin{array}{cc}
                            0& \tilde{B}_1   \\
                             \tilde{B}_1^tI_{1,1} &0 \\
                           \end{array}
                         \right)\Psi^{-1}\dd z=\lambda^{-1} \left(
                           \begin{array}{cc}
                            0& \hat{B}_1   \\
                             \hat{B}_1^tI_{1,1} &0 \\
                           \end{array}
                         \right)\dd z,
\]												
with $\hat{B}_1=\Psi_1\tilde{B}_1\Psi_2^{-1}=\Psi_1\tilde{B}_1\Psi_2^{t}$. So we have
$
\hat{B}_1\hat{B}_1^t=\Psi_1\tilde{B}_1\Psi_2^{t}\Psi_2\tilde{B}_1\Psi_1^{t}=\hat{\gamma}\Psi_1{\bf E}\Psi_1^t.
$
     Then \eqref{eq-b1 of a} follows directly.  And \eqref{eq-b1 of H-A} follows from \eqref{eq-conjugation invariant}.
                         \end{proof}

Note that the isotropic condition $B_1 B_1^t=0$  is equivalent to
the pair of equations $\langle Y_z,Y_z\rangle=\langle\hat{Y}_z,\hat{Y}_z\rangle=0$. So if a non-constant harmonic map $f$ is isotropic, by Theorem  4.8 of \cite{Ma2006}, $Y$ and $\hat{Y}$  form a pair of adjoint Willmore surfaces. Then one has
(compare also \cite{Helein}, \cite{Helein2} and \cite{Xia-Shen}):
\begin{theorem}\cite{Ma2006}, \cite{Helein2} Let $f_h=Y\wedge \hat{Y}$ be an isotropic harmonic map. Then $Y$ and $\hat{Y}$  form a pair of adjoint Willmore surfaces. Moreover, set
\[
B_1=(b_1\ b_2)^t \hbox{ with } b_1,b_2\in\mathbb{C}^{n+2}.
\]
 Then $Y$ is immersed at the points $(b_1^t+b_2^t)(\bar{b}_1+\bar{b}_2)>0$ and $\hat{Y}$ is immersed  at the points  $(b_1^t-b_2^t)(\bar{b}_1-\bar{b}_2)>0$. Especially, when $[Y]$ and $[\hat{Y}]$ are in $\SSS^3$, they are a pair of dual Willmore surfaces.
\end{theorem}

\begin{theorem} \label{thm-normalized potential}(\cite{Helein}, \cite{Helein2}, \cite{Xia-Shen}) Let $f=Y\wedge \hat{Y}$ be an harmonic map with  normalized potential
\[
\eta=\lambda^{-1}\left(
                   \begin{array}{cc}
                    0 & \hat{B}_1 \\
                     -\hat{B}_1^tI_{1,1} & 0\\
                   \end{array}
                 \right)\dd z
\]
                 satisfying \eqref{eq-b1 of H-A}. Then either $f$ is a half$-$isotropic harmonic map (and $Y$ is a Willmore surface), or
\[
\hat{B}_1=\left(
              \begin{array}{cc}
                \hat{b}_1 &
                -\hat{b}_1  \\
              \end{array}
            \right)^t.
\]
\end{theorem}

\begin{proof}
By the DPW construction, an extended frame $F$ of $f$ is derived from the decomposition
$F=F_-\cdot F_+,$ for some $F_-$  such that $F_-^{-1}\dd F_-=\eta,  F_-(0,\lambda)=I.$
Assume that
$F_+=\sum_{j=0}\lambda^{j}F_{+j}$ is the Taylor expansion of $F_+$ with respect to $\lambda\in\mathbb{C}$. So
$
F_{+0}= \textup{diag} \left(
                   F_{+01}
                     , F_{+02}
                 \right),  $$ \hbox{ with } F_{+01}\in SO(1,1,\mathbb{C}), F_{+02}\in SO(n+2,\mathbb{C}).$
Then let
\[F^{-1}\dd F=\lambda^{-1}\alpha_1+\alpha_0+\lambda\alpha_{-1}\ \hbox{ with }\
\alpha_1= \left(
                           \begin{array}{cc}
                            0&  B _1   \\
                             {B}_1^tI_{1,1} &0 \\
                           \end{array}
                         \right)\dd z.
\]
We have
\[
\left(
                           \begin{array}{cc}
                            0&  B_1   \\
                             \hat{B}_1^tI_{1,1} &0 \\
                           \end{array}
                         \right)=F_{+0}^{-1}\left(
                           \begin{array}{cc}
                            0& \hat{B}_1   \\
                             \hat{B}_1^tI_{1,1} &0 \\
                           \end{array}
                         \right)F_{+0}.
\]												
So
$B_1=F_{+01}^{-1}\hat{B}_1F_{+02}$. By \eqref{eq-conjugation invariant},
$B_1$ satisfies \eqref{eq-b1 of H-A}. The rest follows from Theorem \ref{thm-mc form2}.
\end{proof}

Concerning holomorphic potentials, by similar methods, we have
 \begin{theorem}\label{thm-holo-potent-half}
Let $f:\mathbb{D}\rightarrow SO^+(1,n+3)/(SO^+(1,1)\times SO(n+2))$ be a non-constant harmonic map, with an extended frame $F(z,\bar{z},\lambda)$ . Let
 \[
\Xi=C^{-1}\dd C=\sum_{j=-1}^{+\infty}\lambda^{j}\xi_{j} \dd z
\]
be a holomorphic potential of $f$ given by a holomorphic frame $C=F\cdot V_+$. Assume that
\[
\xi_{-1}= \left(
                   \begin{array}{cc}
                    0 & \hat{B}_1 \\
                     -\hat{B}_1^tI_{1,1} & 0\\
                   \end{array}
                 \right).
\]							
Then
\begin{enumerate}
\item
 $f=Y\wedge\hat{Y}$ is an isotropic harmonic map if and only if
 \begin{equation}\label{eq-b1 of a-xi}\hat{B}_1\hat{B}_1^t=0.
\end{equation}
Moreover, $[Y]$ and $[\hat{Y}]$ forms a pair of dual (S-)Willmore surfaces if and only if $rank(
\hat{B}_1)=1$.
\item If $f$  is a half$-$isotropic harmonic map, then $\hat{B}_1$ satisfies
$\hat{B}_1\hat{B}_1^t=\widehat{\gamma}{\bf E}.$ Conversely, if  $\hat{B}_1$ satisfies $\hat{B}_1\hat{B}_1^t=\widehat{\gamma}{\bf E}$, then either $f$ is a half$-$isotropic harmonic map, or
\[
\hat{B}_1=\left(
              \begin{array}{cc}
                \hat{b}_1&                -\hat{b}_1 \\
              \end{array}
            \right)^t.
\]
In the latter case, $f$ is not a half isotropic harmonic map. But if $\widehat{\gamma}\equiv 0$, then  $\hat{Y}$ is M\"obius equivalent to  a minimal surface in $\mathbb{R}^{n+2}$ and $\tilde {f}:=\hat Y\wedge {Y}$  is the isotropic harmonic map given by $\hat{Y}$  and its dual surface $Y$.
\end{enumerate}
\end{theorem}

\section{Generalized Bj\"{o}rling's Problem for Willmore surfaces in $\SSS^{n+2}$}  \label{section8}
 We are now in a position to  solve a generalization of
the Bj\"{o}rling problem for all Willmore surfaces  in $\SSS^{n+2}$, with or without umbilic points.

\subsection{The $\SSS^{3}$ case.}
To address  Willmore surfaces with umblilic points in $\SSS^{3}$, one needs to consider half$-$isotropic harmonic maps instead of the isotropic ones, because, at umbilic points, $[Y]$ and $[\hat Y]$ may coincide and then $Y\wedge \hat Y$ is not well-defined.
In the half$-$isotropic case, if we only prescribe $Y$, $\hat Y$ and $\psi$, we will not have enough information on the tangent
plane of $\hat Y$ to generate a unique solution.
 A solution is to additionally prescribe the $v$ derivative $\hat Y_v$ along the curve.

\begin{theorem}\label{thm-Bjorling-G}
Let $\psi_0=\psi_0(u):\mathbb{I}\rightarrow \SSS^4_1$ denote a non-constant real analytic sphere congruence from $\mathbb{I}$ to $\SSS^3$, with a real analytic enveloping curve $[Y_0]$ and $u$ being the arc-parameter of $Y_0:\mathbb{I}\rightarrow \mathcal{C}_+^{4}\subset\mathbb{R}^5_1$. Let $\hat{Y}_0:\mathbb{I}\rightarrow \mathcal{C}_+^{4}$ be a real analytic map such that $\langle \psi_0, \hat{Y}_0\rangle=0$ and  $\langle Y_0, \hat{Y}_0\rangle=-1$. Let $\gamma_{12}:\mathbb{I}\rightarrow \real$ be a real analytic function.

 Then there exists a unique Willmore surface $y:\Sigma\rightarrow \SSS^3$, with conformal Gauss map $\psi$, $\Sigma$ some simply connected open subset containing $\mathbb{I}$ and $z=u+iv$ a complex coordinate of $\Sigma$, such that:
  \begin{enumerate}
\item  The canonical lift $Y$ of $y$ satisfies
$Y|_{\mathbb{I}}=Y_0$;
\item The conformal Gauss map $\psi$ satisfies $\psi|_{\mathbb{I}}=\psi_0$ and  $\langle\psi_{v}|_{\mathbb{I}}, \hat{Y}_{0}\rangle=-\gamma_{12}.$
\end{enumerate}\end{theorem}

Theorem \ref{thm-Bjorling-G} is a straightforward corollary of the following
\begin{theorem}\label{thm-Bjorling-BP-G}
We retain the assumptions and  notations   in Theorem \ref{thm-Bjorling-G}. Choose two real analytic unit vector fields $P_1$ and $P_2$ on $\mathbb{I}$ such that
\[
Y_{0u}=P_1\mod Y_0,\ ~P_2\perp\{\psi_0,Y_0,\hat{Y}_0,P_1\}\ \hbox{ and }~\det(Y_0,\hat{Y}_0,P_1,P_2,\psi_0)=1.
\]
There exist real analytical functions
  $\mu_1=\mu_1(u)$, $\rho_1=\rho_1(u)$,  $\rho_2=\rho_2(u)$, $k_1=k_1(u)$,
	$k_2=k_2(u)$ and 	$\gamma_{11}=\gamma_{11}(u)$  on $\mathbb{I}$ such  that
\begin{equation}\label{eq-moving-y-G}
\left\{\begin {array}{lllll}
Y_{0u}=-\mu_1Y_0+P_1,\\
\hat{Y}_{0u}=\mu_1\hat{Y}_0+\rho_1P_1+\rho_2P_2+4\gamma_{11}\psi_0,\\
P_{1u}=\mu_2P_2+2k_1\psi_0+\hat{Y}_0+\rho_1 Y_0,\\
P_{2u}=-\mu_2P_1-2k_2\psi_0+\rho_2 Y_0,\\
\psi_{0u}=-2k_1P_1+2k_2P_2  + 4\gamma_{11} \hat Y_0,\\
\end {array}\right.
\end{equation}
holds. Set $\mu=\mu_1+i\mu_2,$ $k=k_1+ik_2$, $\rho=\rho_1+i\rho_2$  and $\gamma_1=\gamma_{11}+i\gamma_{12}$.
For a real analytic function $x(u)$ on $\mathbb I$, denote its analytic extension
to a simply connected open subset containing ${\mathbb I}$ by $x(z)$.
Consider the holomorphic potential
 \[
 \Xi=\left(\lambda^{-1}\mathcal{A}_1 + \mathcal{A}_0 +\lambda\mathcal{A}_{-1} \right) \dd z,
\]
with
\beqas
\mathcal{A}_0=\bbar {cc} A_1 & 0 \\ 0 & A_2 \ebar, \quad
\mathcal{A}_1 = \bbar{cc} 0 & B_1 \\ -B_1^t I_{1,1} & 0 \ebar, \quad
\mathcal{A}_{-1}(z) = \overline{\mathcal{A}_1(\bar z)},\\
 A_1(z) = \bbar{cc} 0 & \mu_1(z) \\ \mu_1(z) & 0 \ebar, \quad
A_2(z) =\bbar{ccc}   0 &  -\mu_2(z)  & -2k_1(z)  \\
       \mu_{2}(z) &0 & 2k_2(z)    \\
       2k_1(z) & -2k_2(z) & 0  \ebar,  \\
B_1(z) = \frac{1}{2\sqrt{2}} \bbar{ccc}
        1+\rho(z) &  -i-i\rho(z)  & 4 \gamma_1\\
        1-\rho(z) &  -i+i\rho(z) &  -4\gamma_1 \ebar. \\
\eeqas
By DPW, Theorem \ref{thm-holo-potent-half}, the potential $\Xi$ provides a half$-$isotropic harmonic map, together with a unique Willmore surface  $y:\Sigma\rightarrow \SSS^3$, with conformal Gauss map $\psi$, $\Sigma$ some simply connected open subset containing $\mathbb{I}$ and $z=u+iv$ a complex coordinate of $\Sigma$,  such that the canonical lift $Y$ of $y$ satisfy $
Y|_{\mathbb{I}}=Y_0$.
Then
$\psi|_{\mathbb{I}}=\psi_0$ and $\langle\psi_{v}|_{\mathbb{I}}, \hat{Y}_{0}\rangle=-\gamma_{12}$. \end{theorem}

\begin{proof} The proof can be taken verbatim from the proof of Theorem \ref{thm-Bjorling-BP}, with the only difference being that
 here the function $\gamma_1$ in the matrix $B_1(z)$ is allowed to be non-zero.
The real part of $\gamma_1(u)$ can be read off from \eqref{eq-moving-y-G}. But the imaginary part of  $\gamma_1(u)$ stays unknown, and we prescribe this as $\gamma_{12}(u)$. The rest is the same as the proof of Theorem \ref{thm-Bjorling-BP}.  The equality
$\langle\psi_{v}|_{\mathbb{I}}, \hat{Y}_{0}\rangle=-\gamma_{12}(u)$ follows from the fact that for a Willmore surface $Y$
 with a half$-$isotropic harmonic map $Y\wedge \hat Y$, $\gamma_1=\frac{1}{2}\langle \hat Y_z,\psi\rangle$.
\end{proof}

The potential $\Xi$ defined in the above theorem is also called the \emph{boundary potential} of the harmonic map.

\begin{remark}  \begin{enumerate}
\item In contrast to the fully isotropic framework, here one can, for any Willmore surface $y$,  locally choose a solution $\mu$ to the equation $\mu_z-\frac{\mu^2}{2}-s=0$ with $\mu$ finite. Then one obtains a half$-$isotropic harmonic map $Y\wedge \hat Y$.
Thus, the above theorem holds locally for any Willmore surface in $S^3$.

\item Choose $\hat Y_0$ to be an enveloping curve of $\psi_0$, pointwisely different from $Y_0$, and set $\gamma_{12}\equiv0$. Then we re-obtain Theorem \ref{thm-Bjorling-BP}.
\item An extremal case is that $Y_0(\mathbb{I})$ is an umbilic curve of $Y$.
 For example, the Willmore tori constructed by Babich and Bobenko \cite{bab-bob}  contain an umbilic curve at the
intersection of the upper and lower hemisphere models of $\HHH^3$.
We can construct any Willmore surface with a line of umbilics with the following characterization
(see Figure \ref{umbilicexamples}):
\end{enumerate}
\end{remark}
\begin{corollary}
We retain the assumptions and notations of Theorem \ref{thm-Bjorling-G} and \ref{thm-Bjorling-BP-G}. Then $Y_0(\mathbb I)$ is  an umbilic curve of $Y$  if and only if $k_1=k_2\equiv0$ on $\mathbb I$.
\end{corollary}

\begin{figure}[ht]
\centering
$
\begin{array}{ccc}
\includegraphics[height=28mm]{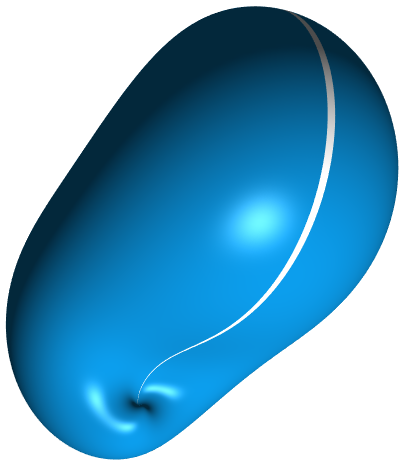} \quad &
\includegraphics[height=28mm]{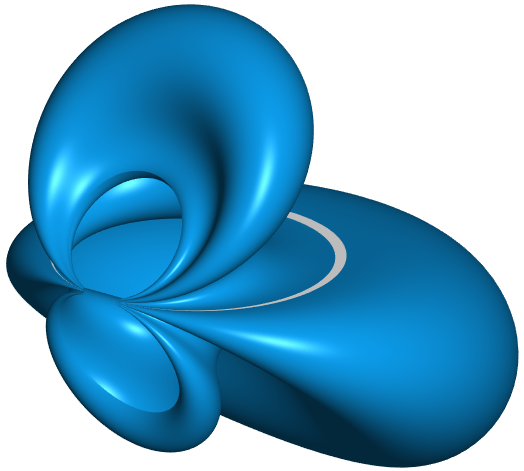}   & \quad
\includegraphics[height=28mm]{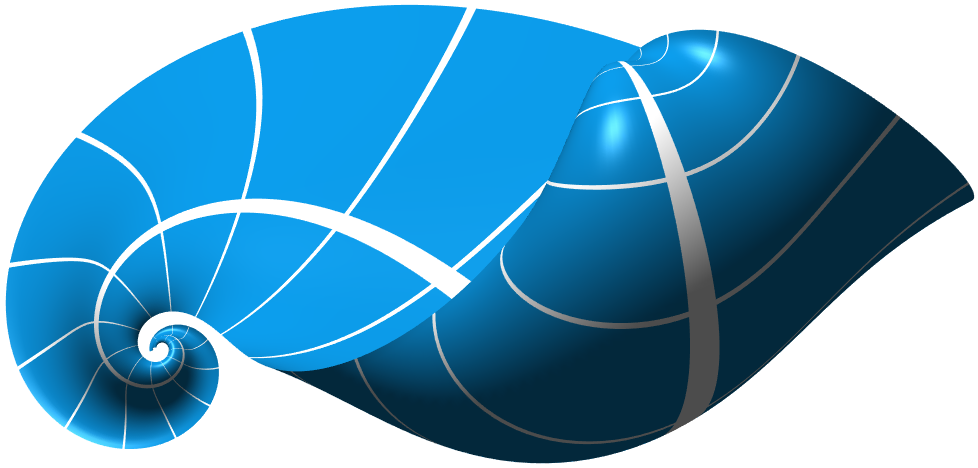}
\end{array}
$
\caption{Willmore surfaces with umbilic lines (Example \ref{example_umbilic}).}
\label{umbilicexamples}
\end{figure}
\begin{example} \label{example_umbilic}
Three examples with lines of umbilics are computed and displayed in Figure \ref{umbilicexamples}.
From left to right, the Bj\"orling data are:
 $(\mu, k, \rho, \gamma_1)=(1+i, 0, 1+i, 1)$,
 $(\mu, k, \rho, \gamma_1)=(0, 0, 0, i)$ and
  $(\mu, k, \rho, \gamma_1)=(\sin u+e^{0.1u}+i(-1+0.5u+\sin u), 0, \cos 3u+i(1+0.3u), 1+0.2u+2i(\sin u +0.6u))$.
\end{example}

\begin{example}  \label{example-n1} Similar to Example \ref{example1},
let us consider a Willmore surface in $\SSS^3$ containing the circle $(\cos u, \sin u, 0,0)$, with a lift $Y=(1,\cos u, \sin u, 0,0)$, $\hat Y = (1/2)(1,-\cos u, -\sin u, 0,0)$ and a free function $\gamma_{12}$. Then similar to discussions in Example \ref{example1}, we have
\[
P_1 = (0,-\sin u, \cos u, 0,0),\ \psi = -E_3\sin\theta +  E_4 \cos\theta, \
P_2 = -E_3\cos\theta - E_4\sin\theta,
\]
 where $\theta$ is any real analytic map $\real \to \real$. We also have
$\rho_1 = -1/2$, $k_2 = \theta^\prime/2$ and $\rho_2 =\gamma_{11}=\mu_2 = k_1 = 0$.
So we can say that all solutions
corresponding to the pair $Y$ and $\hat Y$ above are obtained from a choice of two functions $\theta$ and $\gamma_{12}$ with the boundary potential given by the data:
\[
(\mu,k,\rho, \gamma_1) = (0,i\theta^\prime/2,-1/2,i\gamma_{12}).
\]
Three examples are shown at  Figure \ref{circleexamples2}, the first with no umbilics on the circle,
the second with two umbilics on the circle, and the last with a line of umbilics. The
Bj\"orling data are, in order,
$(\mu, k, \rho, \gamma_1) = (0,i/2, -1/2, i\sin 4u)$,
$(\mu, k, \rho, \gamma_1) = (0, i\sin u, -1/2, i)$,
$(\mu, k, \rho, \gamma_1) = (0, 0, -1/2, i\cos u)$.
\end{example}

\begin{figure}[ht]
\centering
$
\begin{array}{ccc}
\includegraphics[height=34mm]{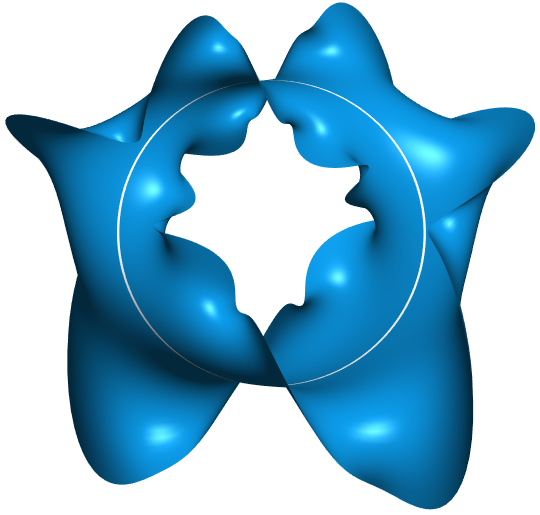} \quad & \quad
\includegraphics[height=34mm]{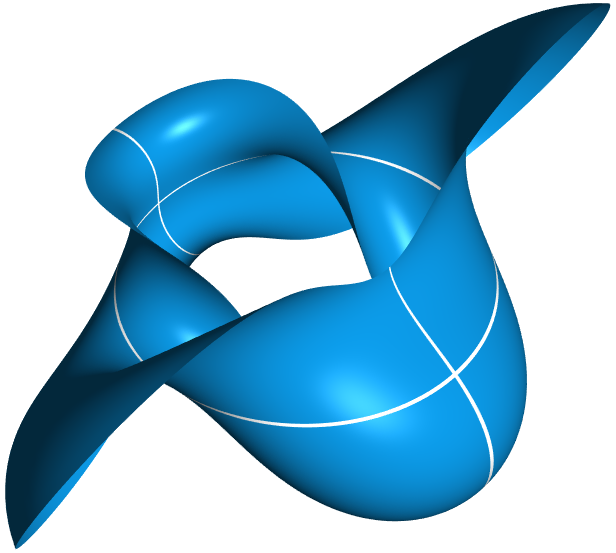} \quad   & \quad
\includegraphics[height=34mm]{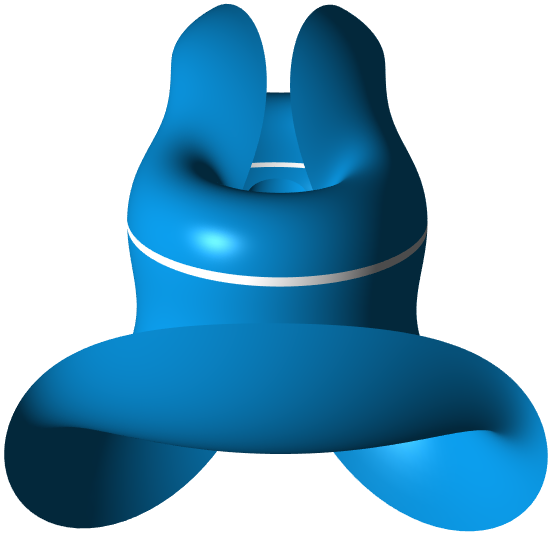}
\end{array}
$
\caption{Willmore surfaces containing a circle (Example \ref{example-n1}).}
\label{circleexamples2}
\end{figure}

\subsection{Generalized Bj\"{o}rling's problem for Willmore surfaces in $\SSS^{n+2}$.}
The above result can be generalized to Willmore surfaces in $\SSS^{n+2}$.
We write down the solution to the generalized Bj\"{o}rling problem for the half-isotropic harmonic maps associated to a Willmore surface in $\SSS^{n+2}$ as follows. In higher codimension, it will be convenient to use $Y_0\wedge\hat Y_0\wedge P_{01}\wedge P_{02}$ to represent  sphere congruences. We refer to \cite{Her1} for the representation of sphere congruences in $\SSS^{n+2}$ (See also \cite{BPP}, \cite{Ma2006} for a discussion of mean curvature spheres).

 \begin{theorem}\label{thm-Bjorling-h0}
Let $\Phi_0=Y_0\wedge\hat Y_0\wedge P_{1}\wedge P_{2}:\mathbb{I}\rightarrow SO^+(1,n+3)/(SO^+(1,3)\times SO(n))$ denote a real analytic sphere congruence from $\mathbb{I}$ to $\SSS^{n+2}$ such that
  \begin{enumerate}
\item $Y_0:\mathbb{I}\rightarrow \mathcal{C}_+^{n+3}\subset\mathbb{R}^{n+4}_1$ is a real analytic curve with arc-parameter $u$ and $[Y_0]$ is an enveloping curve of $\Phi_0$;
\item The real analytic map $\hat{Y}_0:\mathbb{I}\rightarrow \mathcal{C}_+^{n+3}$ satisfies  $\langle Y_0, \hat{Y}_0\rangle=-1$;
\item The real analytic map $\zeta:\mathbb{I}\rightarrow \real^{n+4}_1$ is perpendicular to $\{Y_0,\hat Y_0, P_{1}, P_{2}\}$.
\end{enumerate}

 Then there exists a unique half$-$isotropic harmonic map  $Y\wedge \hat Y :\Sigma\rightarrow SO^+(1,n+3)/(SO^+(1,1)\times SO(n+2))$ and a unique Willmore surface $y=[Y]:\Sigma\rightarrow \SSS^{n+2}$, with conformal Gauss map $\Phi$, $\Sigma$ some simply connected open subset containing $\mathbb{I}$ and $z=u+iv$ a complex coordinate of $\Sigma$, such that:
  \begin{enumerate}
\item  The canonical lift $Y$ of $y$ satisfies
$Y|_{\mathbb{I}}=Y_0$.
\item The map $\hat Y$ satisfies $\hat Y|_{\mathbb{I}}=\hat Y_0,\quad \hat Y_{v}|_{\mathbb{I}}=\zeta \mod\{Y_0,\hat Y_0, P_{1}, P_{2}\}$;
\item  The conformal Gauss map $\Phi$ of $y$ satisfies
$\Phi|_{\mathbb{I}}=\Phi_0$.
\end{enumerate}\end{theorem}
\begin{proof} Assume that
the real analytic maps $P_{1},P_{2}:\mathbb{I}\rightarrow S^{n+3}_1$ satisfies
 \[\hbox{$P_1 =Y_{0u} \mod Y_0$, $\{P_{1}, P_{2}\}\perp \{Y_0,\hat Y_0\}$ and $P_{1}\perp P_{2}$},\]
 and $\{\psi_{01}, \dots \psi_{0n}\}$ is a real analytic orthonormal basis of the orthogonal complement of $\{P_{1}, P_{2},  Y_0,\hat Y_0\}$.
The proof follows from the higher co-dimensional
analogue of Theorem \ref{thm-Bjorling-BP-G}, the statement and
proof of which generalize,
replacing $\psi_0$ of Theorem \ref{thm-Bjorling-BP-G} with
$\psi_{01}, \dots \psi_{0n}$, substituting the equations
\begin{equation}\label{eq-eq-moving-y-GH}
\left\{\begin {array}{lllll}
Y_{0u}=-\mu_1Y_0+P_1,\\
\hat{Y}_{0u}=\mu_1\hat{Y}_0+\rho_1P_1+\rho_2P_2+4\sum_j\gamma_{j1}\psi_{0j},\\
P_{1u}=\mu_2P_2+2\sum_{j=1}^nk_{j1}\psi_{0j}+\hat{Y}_0+\rho_1 Y_0,\\
P_{2u}=-\mu_2P_1-2\sum_{j=1}^nk_{j2}\psi_{0j}+\rho_2 Y_0,\\
\psi_{0ju}=\sum_{l=1}^n b_{jl1}\psi_{0l}-2k_{j1}P_1+2k_{j2}P_2+4\gamma_{j1} Y,\ 1\leq j \leq n,\\
\end {array}\right.
\end{equation}
for the equations \eqref{eq-moving-y-G},
and writing down the corresponding Maurer-Cartan form for the associated
frame, which has the same form as \eqref{eq-hm-m-c-h}. Note in this case $\gamma_{j2}$ is given by $\zeta$, i.e., $\gamma_{j2}=\frac{1}{4}\langle\zeta, \psi_{0j}\rangle$.
We leave these details to the interested reader.
\end{proof}

To adapt Theorem \ref{thm-Bjorling-h0} to the isotropic case, we need only add the
assumption that $\zeta$ has the same length as $\sum_j\gamma_{j1}\psi_{0j}$ in \eqref{eq-eq-moving-y-GH},
 which is to ensure $\hat Y$ is also conformal in $z$. This is equivalent to prescribing the  mean curvature sphere of $\hat Y$
(in addition to that of $Y$).  See the proof of the following theorem for the details. Note  that in general these two mean curvature spheres are different, which is also the geometric reason  why  two mean curvature spheres are needed to solve the Bjorling problem
in the general case.
 \begin{theorem}\label{thm-Bjorling-h1}
Let $\Phi_0=Y_0\wedge\hat Y_0\wedge P_{1}\wedge P_{2},\ \hat\Phi_0=\hat Y_0\wedge Y_0\wedge \hat{P}_{1}\wedge \hat{P}_{2}:\mathbb{I}\rightarrow SO^+(1,n+3)/(SO^+(1,3)\times SO(n))$ denote two real analytic sphere congruences from $\mathbb{I}$ to $\SSS^{n+2}$ such that
  \begin{enumerate}
\item $Y_0:\mathbb{I}\rightarrow \mathcal{C}_+^{n+3}\subset\mathbb{R}^{n+4}_1$ is a real analytic curve with arc-parameter $u$ and $[Y_0]$ is an enveloping curve of $\Phi_0$;
\item The real analytic map $\hat{Y}_0:\mathbb{I}\rightarrow \mathcal{C}_+^{n+3}$ satisfies  $\langle Y_0, \hat{Y}_0\rangle=-1$. And it is an enveloping curve of $\hat\Phi_0$ at the points it is immersed.
 \end{enumerate}

 Then there exists a unique isotropic harmonic map  $Y\wedge \hat Y :\Sigma\rightarrow SO^+(1,n+3)/(SO^+(1,1)\times SO(n+2))$ and a unique Willmore surface $y=[Y]:\Sigma\rightarrow \SSS^{n+2}$, with   an adjoint transform $\hat y=[\hat Y]$, $\Sigma$ some simply connected open subset containing $\mathbb{I}$ and $z=u+iv$ a complex coordinate of $\Sigma$, such that:
  \begin{enumerate}
\item  The canonical lift $Y$ of $y$ satisfies
$Y|_{\mathbb{I}}=Y_0$;
\item The map $\hat Y$ satisfies $\hat Y|_{\mathbb{I}}=\hat Y_0$;
\item  The conformal Gauss map $\Phi$, $\hat \Phi$ of $y$ and $\hat y$ satisfies $\Phi|_{\mathbb{I}}=\Phi_0$, $\hat\Phi|_{\mathbb{I}}=\hat\Phi_0$.
\end{enumerate}\end{theorem}

\begin{proof}
Since $\hat Y_0$ is an enveloping curve of $\hat\Phi_0$, $\hat Y_{0u}\in Span\{\hat Y_0, Y_0, \hat P_{1}, \hat P_{2}\}$. So we can assume that $\hat Y_{0u}= a\hat P_1$ and $\zeta= a\hat P_2 \mod \{Y_0,\hat Y_0, P_{1}, P_{2}\}$. Applying Theorem \ref{thm-Bjorling-h0}, we finish the proof.
\end{proof}

Restricting to the case of  a pair of dual S-Willmore surfaces in $\SSS^{n+2}$, we obtain the following

 \begin{theorem}\label{thm-Bjorling-h}
Let $\Phi_0:\mathbb{I}\rightarrow SO^+(1,n+3)/(SO^+(1,3)\times SO(n))$ denote a non-constant real analytic sphere congruence from $\mathbb{I}$ to $\SSS^{n+2}$, with enveloping curves $[Y_0]$ and $[\hat{Y}_0]$ such that $\langle Y_0, Y_0 \rangle=\langle \hat{Y}_0, \hat{Y}_0\rangle=0$, $\langle Y_0, \hat{Y}_0\rangle=-1$, and $u$ is  the arc-length parameter of $Y_0$. Then there exists a unique pair of dual (S-Willmore) Willmore surfaces $y,\hat{y}:\Sigma\rightarrow \SSS^{n+2}$, with $\Sigma$ some open subset containing $\mathbb{I}$, such that
  \begin{enumerate}
\item  There exist lifts $Y$, $\hat Y$ of $y,\hat y$ such that
$Y|_{\mathbb{I}}=Y_0$,  $\hat Y|_{\mathbb{I}}=\hat Y_0$;
\item  The conformal Gauss map $\Phi$ of $y$  satisfies $\Phi|_{\mathbb{I}}=\Phi_0$.
\end{enumerate}\end{theorem}
\vspace{2 mm}

{\small{\bf Acknowledgements}\ \ We would like to thank the referee for helpful suggestions
that improved the results of this paper. The second named author was supported by the NSFC Project No. 11571255.}

\def\refname{Reference}

{\small\ \ \vspace{2mm}

David Brander

Department of Applied Mathematics and Computer Science,
Technical University of Denmark,
Matematiktorvet, Bdg. 303 B,
DK-2800 Kgs. Lyngby, DENMARK.

{\em E-mail address}: dbra@dtu.dk\vspace{2mm}

Peng Wang

Department of Mathematics, Tongji University, Siping Road 1239, Shanghai, 200092, P. R. China.

{\em E-mail address}: {netwangpeng@tongji.edu.cn}
 }

\end{document}